\documentclass[11pt, reqno]{amsart}
\usepackage{amsmath, amsthm, amscd, amsfonts, amssymb, graphicx, color}
\usepackage[bookmarksnumbered, colorlinks, plainpages]{hyperref}

\chardef\bslash=`\\ 

\newtheorem{thm}{Theorem}[section]
\newtheorem{cor}[thm]{Corollary}
\newtheorem{lem}[thm]{Lemma}
\newtheorem{prop}[thm]{Proposition}

\theoremstyle{definition}
\newtheorem{defn}[thm]{Definition}

\theoremstyle{remark}
\newtheorem{rem}[thm]{Remark}

\newcommand{\thmref}[1]{Theorem~\ref{#1}}

\newcommand{\lemref}[1]{Lemma~\ref{#1}}
\newcommand{\propref}[1]{Proposition~\ref{#1}}
\newcommand{\corref}[1]{Corollary~\ref{#1}}
\newcommand{\defnref}[1]{Definition~\ref{#1}}

\newcommand{\B}{\mathcal{B}}

\newcommand{\F}{\mathcal{F}}
\newcommand{\G}{\mathcal{G}}

\newcommand{\m}{\bf m}

\newcommand{\eval}[2][\right]{\relax
  \ifx#1\right\relax \left.\fi#2#1\rvert}

\def\Z{{\mathbf{Z}}}

\def\t{{\mathbf{T}}}

\def\defequal{\stackrel{\mathrm{def}}{=}}

\def\<{\langle}
\def\>{\rangle}

\def\tfae{the following conditions are equivalent:}

\title{Cartan subalgebras in W$^*$-algebras}
\author{Jean Renault}
\address{Institut Denis Poisson (UMR 7013)\\
  Universit\'e d'Orl\'eans et CNRS \\ 45067
  Orl\'eans Cedex 2, FRANCE}
\email{jean.renault@univ-orleans.fr}
\keywords{Boolean algebras, Inverse semigroups, Groupoids, Cartan subalgebras.}
\subjclass{Primary 37D35; Secondary 46L85.}

\begin{document}

\vskip5mm
\begin{abstract} This article presents a proof of the Feldman-Moore theorem on Cartan subalgebras in W*-algebras based on the non-commutative Stone equivalence between Boolean inverse semigroups and Boolean groupoids. The proof is decomposed into two parts. The first part writes the W*-algebra as the W*-algebra of the Weyl twist of the Cartan subalgebra. The second part is an extension to separable measure inverse semigroups of the realization of a separable measure algebra by a standard measure space.
\end{abstract}

\maketitle
\markboth{Renault}
{Cartan}

\renewcommand{\sectionmark}[1]{}

\section*{Introduction}

In the theory of operator algebras, a Cartan subalgebra is a maximal abelian self-adjoint subalgebra with additional properties (\defnref{W*-Cartan}). The main point of  \cite{kum:diagonals, ren:cartan} is that there is a canonical twisted groupoid attached to a Cartan pair of C*-algebras and that the Cartan pair constructed from this twisted groupoid is isomorphic to the initial Cartan pair. Not surprisingly, a similar result holds for Cartan pairs of W*-algebras. In fact, the study of W*-Cartan pairs \cite{fm:relations II} preceded  and motivated the study of C*-Cartan pairs; moreover, it is easier, because it suffices to consider the partial isometries which normalize the subalgebra rather than arbitrary normalizers. Since the partial isometry normalizers form an inverse semigroup, it is natural to appeal to the theory of inverse semigroups to study the Cartan pair. This is the road followed by A.~Donsig, A.~Fuller and D.~Pitts in \cite{dfp:cartan}. The close kinship of inverse semigroups and groupoids has been known for a long time and under various forms (see in particular \cite{ren:approach, pat:gpd, ste:isga}). One of its most basic expressions is that the bisections of a groupoid form an inverse semigroup. One can ask for a converse as in \cite[Remark III.2.4]{ren:approach}. This question is studied in \cite{pat:pb} and \cite[Section 4.3]{pat:gpd}. It turns out that it has a very satisfactory answer, given by M.~Lawson in \cite{law:stone,law:stone2,law:stone3} and in his work with D.~Lenz \cite{ll:pseudogroups}, for a particular class of inverse semigroups, namely the Boolean inverse semigroups (\defnref{boolean isg}). Indeed the category of Boolean inverse semigroups is equivalent to the category of Boolean groupoids (\thmref{equivalence}), where a Boolean groupoid is an \'etale groupoid whose unit space is Boolean (these groupoids are called ample in \cite{pat:gpd}). Lawson's beautiful idea is that Boolean inverse semigroups generalize Boolean algebras,  Boolean groupoids generalize Boolean spaces and that Stone's duality can be extended in this new framework. However, in order to establish a non-commutative Stone duality, Lawson and Lenz had to restrict the class of objects or arrows. The problem is the right definition of the arrows in the category of groupoids. Motivated by the theory of C*-algebras, groupoid correspondences (see \cite[Section 2]{ren:rieffel}) are good candidates but invertible correspondences are groupoid equivalences and not usual isomorphisms. In the present case of Boolean groupoids, the suitable arrows are the multiplier actions, which are a particular case of groupoid correspondences. We give their definition in the Appendix A and refer the reader to the recent preprint \cite{tay:functoriality} for further information on this notion which appears there under the name of actor. The functors Bis and Germ establishing the equivalence of category are covariant and give an equivalence rather than a dual equivalence. The results of Appendix A are essentially contained in \cite{bem:isa}. Since a W*-Cartan pair gives a twisted Boolean inverse semigroup, the equivalence has to be extended to the twisted categories. The main result of Section 2 (\thmref{abstract}) is a simple transposition of \cite[Theorem 5.6]{ren:cartan} (see also \cite[Theorem 1.2]{raa:cartan}). Its final step is even easier since it is the identification of two GNS representations. Comparing with the inverse semigroup road followed in \cite{dfp:cartan}, one can see the advantage of introducing groupoids and their convolution algebras.

Stone's duality can be split into three levels: firstly Boolean algebras and Boolean spaces, secondly complete Boolean algebras and Stonian spaces, and thirdly measure algebras and measure spaces. The same holds for the non-commutative Stone equivalence; moreover each level corresponds to a class of operator algebras: firstly Boolean inverse semigroups, Boolean groupoids and C*-algebras, secondly complete Boolean inverse semigroups, Stonian groupoids and AW*-algebras, and thirdly measure inverse semigroups, measure Boolean groupoids and W*-algebras. We are interested in the third level. It is different: strictly speaking, the dual of a measure algebra is a hyperstonian space (\cite[D\'efinition 3]{dix:stone}). However the natural objects of measure theory are standard measure spaces and not hyperstonian spaces. When the measure algebra is separable (see for example \cite[Section 40]{hal:measure}), it can be realized as the measure algebra of a standard measure space. We show that a similar result holds for measure Boolean inverse semigroups (\thmref{concrete}): if the W*-algebra $W^*(G)$ of a measure Stonian groupoid $G$ has a separable predual, then $G$ contains (in the sense of the category of Boolean groupoids) a measure Boolean groupoid $H$ which is second countable such that the pairs $(W^*(G),L^\infty(G^{(0)}))$ and $(W^*(H),L^\infty(H^{(0)}))$ are isomorphic, where we have omitted the measures and the possible twists for the sake of legibility.  This step is also a crucial element of the original proof of Feldman and Moore (\cite[paragraph before Proposition 3.2]{fm:relations II}). When we apply this result to the twisted measure Stonian groupoid associated with a Cartan pair, we recover the Feldman-Moore theorem since in this case, $H$ admits a reduction to a Borel set of full measure which is an equivalence relation with countable classes.
\vskip2mm
Here is a brief summary of the content of the paper. Besides this introduction, the paper has three sections. The first section gives the definitions of (twisted) Boolean inverse semigroups and of (twisted) Boolean groupoids, presents briefly the non-commutative Stone equivalence, recalls the construction of the convolution algebra of a twisted \'etale groupoid, and shows that the full C*-algebra of a twisted Boolean groupoid is universal for the representations of its twisted ample semigroup (\thmref{universal}). The second section recalls the definition of a Cartan subalgebra both in the C*-algebra and in the W*-algebra contexts. While the study of \cite{ren:cartan} is restricted to Hausdorff groupoids, it is completed here by some results valid for non-Hausdorff Boolean groupoids. However, \propref{char2}, inspired by \cite[Theorem 3.1]{abcclmr} and \corref{char} give a severe limitation to a direct extension of \cite[Theorem 5.9]{ren:cartan} to the non-Hausdorff case. The main result (\thmref{abstract}) of the section is the realization of a W*-Cartan pair as the W*-Cartan pair of its twist. The third section studies briefly measure inverse semigroups and shows that, under a separability assumption, a Stonian measure groupoid can be replaced by a second countable Boolean measure groupoid without changing its W*-algebra (\thmref{concrete}). The Feldman-Moore theorem is given as a corollary (\corref{fm}). Two appendices complete the paper for the convenience of the reader. Appendix A establishes the equivalence of the category of Boolean inverse semigroup and of the category of Boolean groupoids, first in the untwisted case (\thmref{equivalence}) and then its twisted version (\thmref{twisted equivalence}). These constructions and properties, which are well known in the general setting of \'etale groupoids and inverse semigroups (see for example \cite{exe:isg, be:fell, bem:isa}), take a quite satisfying form in the Boolean case. Appendix B establishes the bijection between the representations of a twisted Boolean inverse semigroup $(\G,{\mathcal S})$ and the representations of the convolution algebra $C_c(G,\Sigma)$ of its twisted groupoid of germs (\thmref{twisted rep}). This is a straightforward result which exists in the literature under various forms (often with extraneous assumptions).
\vskip2mm
The notation about groupoids and inverse semigroups follows essentially that of \cite{ren:approach}. The domain map is denoted by $d$ and the range map by $r$. Since inverse semigroups appear here either as abstract semigroups, or semigroups of partial isometries, or semigroups of bisections, their elements may be written differently: $s,t$ for abstract elements, $u,v$ for partial isometries, and $S,T$ for bisections.
\vskip2mm
I thank Professor Takesaki for suggesting a few years ago that the techniques of \cite{ren:cartan} to study C*-Cartan subalgebras should work as well for W*-Cartan subalgebras, Mark Lawson for sharing his work on non-commutative Stone duality and help, Fred Wehrung and Henrik Kreidler for fruitful discussions and Davit Pitts for useful comments.

\section{Twisted Boolean inverse semigroups and groupoids }

\subsection{Boolean inverse semigroups and Boolean groupoids} Definitions and references about inverse semigroups are recalled in  Appendix A. We denote by $\G^{(0)}$ the set of idempotents of the inverse semigroup $\G$. It is a meet-semilattice for the order $e\le f$ if and only if $ef=e$. One extends this order to $\G$ by defining $s\le t$ if and only if there exists $e\in\G^{(0)}$ such that $s=te$. Assuming that $\G$ has a zero element, we say that a family $(s_i)$ of elements of $\G$ is orthogonal if for all $i\not=j$, $s^*_is_j=s_is^*_j=0$. The following notion, which plays a crucial role in the present work, is anticipated in \cite[Remark III.2.4]{ren:approach}. The related notion of additive inverse semigroup is predominant in \cite{pat:gpd}. The terminology of Boolean inverse semigroup is introduced in \cite{law:stone}, where it is shown that Boolean inverse semigroups are a natural and far-reaching generalization of Boolean algebras. Although various definitions of a Boolean inverse semigroup appear in the literature (see \cite[Section 3.1]{weh:boolean}), we stick to our original notion because it is the most appropriate in this work. 

\begin{defn}\label{boolean isg}  An inverse semigroup $\G$ is called {\it Boolean} if
\begin{enumerate}
 \item $\G^{(0)}$ is a Boolean algebra;
 \item every orthogonal pair $(s,t)$ in $\G$ has a join $s\vee t$ in $\G$.
\end{enumerate}
 \end{defn}
 
 Although we are mostly concerned with the case when $\G^{(0)}$ has a unit, we do not make this assumption here. Thus $\G^{(0)}$ is a generalized Boolean algebra or a Boolean ring in the usual terminology. We choose this definition to match the corresponding definition for C*-algebras.
 
 \begin{defn}
A topological space is called {\it Boolean} if it is locally compact Hausdorff and totally disconnected.
\end{defn}

\begin{defn}\label{boolean gpd}  A {\it Boolean groupoid} is an \'etale groupoid $G$ whose unit space $G^{(0)}$ is a Boolean space.  
\end{defn}

A Boolean groupoid is called an ample groupoid by Paterson in \cite{pat:gpd}. We shall not use his terminology. 
\vskip 2mm
The non-commutative Stone equivalence, proved in Appendix A, is the equivalence of the category of Boolean inverse semigroups, where the arrows are the Boolean inverse semigroup morphisms and the category of Boolean groupoids, where the arrows are the multiplier actions with proper anchor. One associates to a Boolean inverse semigroup $\G$ its groupoid of germs ${\rm Germ}(\G)$ and to a Boolean groupoid $G$ the inverse semigroup of its compact open bisections ${\rm Bis}(G)$. We recall these constructions. The next proposition is proved in Appendix A. The terminology follows \cite{kri:dimension} and \cite[Definition I.2.10]{ren:approach}.

\vskip2mm
\noindent
{\bf Proposition A.1.} {\it Endowed with the set-theoretical product, the set of compact open bisections of a Boolean groupoid $G$ is a Boolean inverse semigroup, called the {\em ample inverse semigroup} of $G$ and denoted by ${\rm Bis}(G)$}.
\vskip2mm

Whenever an inverse semigroup $\G$ acts on a topological space $X$ by partial homeomorphisms, one can define the groupoid of germs $G={\rm Germ}(\G,X)$ of this action (see Paterson \cite[Theorem 3.3.2]{pat:gpd} or Exel \cite[Proposition 4.17]{exe:isg}). Explicitly $G=\G*X/\sim$, where $\G*X$ is the set of pairs $(s,x)$ where $s\in\G$ and $x\in{\rm dom}(s)$ and $(s,x)\sim (t,y)$ if and only if $x=y$ and there exists $e\in \G^{(0)}$ having $x$ in its domain and such that $se=te$. We denote by $[s,x]$ the equivalence class of $(s,x)$. The operations $[s,\varphi_t(x)][t,x]=[st,x]$ and $[s,x]^{-1}=[s^*,\varphi_s(x)]$ turn $G$ into a groupoid. The topology of germs, with basic open sets $S(s,U)=\{[s,x], x\in U\cap{\rm dom}(s)\}$, where $s\in\G$ and $U$ is an open set of $X$, turns it into an \'etale groupoid. If $X$ is a Boolean space, $G$ is a Boolean groupoid. We apply this construction to the action of $\G$ on $\G^{(0)}$ by conjugation: $E\mapsto SES^*$. If $\G^{(0)}$ is the Boolean algebra of compact open sets of the Boolean space $X$, it gives an action of $\G$ on $X$ by partial homeomorphism $\varphi_S:{\rm dom}(S)\to{\rm ran}(S)$.

\begin{defn}\label{germ}
 The {\it groupoid of germs} ${\rm Germ}(\G)$ of a Boolean inverse semigroup $\G$ is defined as the groupoid of its action on the dual space $X$ of the Boolean algebra $\G^{(0)}$.
\end{defn}

\begin{rem}
 Our groupoid of germs ${\rm Germ}(\G)$ is not Paterson's universal groupoid of \cite[Section 4.3]{pat:gpd}, which is the groupoid of germs of the action of $G$ on the spectrum $\Omega$ of the meet-semilattice $\G^{(0)}$. More precisely ${\rm Germ}(\G)$ is the reduction of the universal groupoid of $\G$ to $X$, which is a closed invariant subset of $\Omega$ (cf. \cite[Proposition 4.18]{ste:isga}).
\end{rem}

 \vskip4mm
 \subsection{Twists}
 
The non-commutative Stone equivalence extends to twisted Boolean inverse semigroups and twisted Boolean groupoids. The details can be found in  Appendix A.

 \begin{defn}\label{gpd twist}\cite[\S 2]{kum:diagonals} A {\it twist} over a groupoid $G$ is a central groupoid extension
$$X\times\t\stackrel{i}{\rightarrowtail}\Sigma\stackrel{p}{\twoheadrightarrow}  G$$ 
where $\t$ is the group of complex numbers of modulus one (by definition, the groupoids of the extension have the common unit space $X$). Then, we say that $(G,\Sigma)$ is a {\it twisted groupoid}.
\end{defn} 

When $G$ is a locally compact groupoid, we implicitly assume that $\Sigma$ is a locally compact groupoid, that $X\times\t$ is a locally closed subset of $\Sigma$ and that $p$ is continuous and open. We then say that $(G,\Sigma)$ is a locally compact twisted groupoid. If $G$ is Boolean, we say that $(G,\Sigma)$ is a Boolean twisted groupoid. The notion of open bisection is not suitable for $\Sigma$ which is not \'etale, We replace it by the notion of continuous bisection (\defnref{continuous bisection}). It can be checked that the two notions agree when $G$ is \'etale. The set of compact continuous bisections of $\Sigma$ is a Boolean inverse semigroup, which we still call the {\it ample semigroup} of $\Sigma$ and denote by  ${\rm Bis}(\Sigma)$.
\vskip2mm

In order to define a twist over a Boolean inverse semigroup, we need to interpret the term of ``central extension''. Given a Boolean space $X$, we denote by ${\mathcal T}(X)$ the set of continuous functions $h:{\rm supp}(h)\to\t$, where ${\rm supp}(h)$ is a compact open subset of $X$. It is an inverse semigroup under multiplication, where the support of $hk$ is the intersection of the supports of $h$ and $k$. Note that ${\mathcal T}(X)$ is nothing but the ample semigroup ${\rm Bis}(X\times\t)$ of the trivial group bundle $X\times\t$. It is called the trivial Clifford inverse semigroup over $X$. Every (partial) isomorphism of the Boolean algebra $\B(X)$ of compact open subsets of $X$ has a trivial extension to a (partial) isomorphism of ${\mathcal T}(X)$.

\begin{defn}\label{isg twist} A {\it twist} over a Boolean inverse semigroup $\mathcal G$ is an inverse semigroup extension
\[{\mathcal T}\stackrel{i}{\rightarrowtail}{\mathcal S}\stackrel{p}{\twoheadrightarrow}  \G\]
where ${\mathcal T}$ is the trivial Clifford inverse semigroup over the dual space of ${\G}^{(0)}$ and the action of ${\mathcal S}$ on $\mathcal T$ by conjugation is the trivial extension of its action on  ${\G}^{(0)}$. Then, we say that $(\G,{\mathcal S})$ is a {\it twisted Boolean inverse semigroup}.
\end{defn}

The proof of the next proposition is given in Appendix A.

\vskip2mm
\noindent
\propref{gpdtwist to isgtwist} {\it Let  $(G,\Sigma)$ be a twisted Boolean groupoid. Then $({\rm Bis}(G), {\rm Bis}(\Sigma))$ is a twisted Boolean inverse semigroup.}
\vskip2mm

 In order to show that every twisted Boolean inverse semigroup $(\G,{\mathcal S})$ is isomorphic to $({\rm Bis}(G), {\rm Bis}(\Sigma))$, where $(G,\Sigma)$ is a twisted Boolean groupoid, we introduce the variant of the construction of the groupoid of germs, given in \cite[Proposition 4.12]{ren:cartan}. Let $(\G,{\mathcal S})$ be a twisted Boolean inverse semigroup. We represent $\G^{(0)}$ as the  Boolean algebra $\B(X)$ of a Boolean space $X$ and  the kernel of the extension as the trivial Clifford inverse semigroup ${\mathcal T}(X)$. We recall that  ${\mathcal S}$ acts on $X$ through the action by conjugation and denote by $\varphi_s$ the partial homeomorphism defined by $s\in{\mathcal S}$. We define on ${\mathcal S}*X=\{(s,x): s\in {\mathcal S}, x\in{\rm dom}(s)\}$ the equivalence relations
 $(s,x)\sim (t,y)$ if and only if
  $$x=y\quad{\rm and}\quad\exists\, a,b\in {\mathcal T}(X): a(x)=b(x)=1\quad{\rm and}\quad sa=tb$$
 and
 $(s,x)\approx (t,y)$ if and only if
  $$x=y\quad{\rm and}\quad\exists\, a,b\in {\mathcal T}(X): |a(x)|=|b(x)|=1\quad{\rm and}\quad sa=tb.$$
  The equivalence class of $(s,x)$ for $\sim$ [resp. for $\approx$] is denoted by $[s,x]$ [resp. by $[[s,x]]$]. Again the proof of the next proposition is in  Appendix A.
  
 \vskip2mm
 \noindent 
\propref{isgtwist to gpdtwist} {\it Let $(\G,{\mathcal S})$ be a twisted Boolean inverse semigroup. Then,
\begin{enumerate}
 \item $\Sigma={\mathcal S}*X/\sim$ has a groupoid structure with
$$[s,\varphi_t(x)][t,x]=[st,x],\quad [s,x]^{-1}=[s^*,\varphi_s(x)].$$
 \item $G={\mathcal S}*X/\approx$ has a groupoid structure with
$$[[s,\varphi_t(x)]][[t,x]]=[[st,x]],\quad [[s,x]]^{-1}=[[s^*,\varphi_s(x)]].$$
\item the map $p:\Sigma\to G$ sending $[s,x]$ to $[[s,x]]$
is a surjective groupoid homomorphism such that $p^{(0)}$ is the identity of $X$ and ${\rm Ker}(p)$ is isomorphic to $X\times\t$.
\item $(G,\Sigma)$ is a twisted Boolean groupoid.
\item $(\G,{\mathcal S})$ is isomorphic to $({\rm Bis}(G), {\rm Bis}(\Sigma))$.
\end{enumerate}}
\vskip2mm

The isomorphism associates to $s\in{\mathcal S}$ the bisection $S(s)\defequal\{[s,x], x\in{\rm dom}(s)\}$ of $\Sigma$ and to $p(s)\in\G$ the bisection $S'(s)\defequal\{[[s,x]], x\in{\rm dom}(s)\}$ of $G$.

\begin{defn}\label{twisted groupoid of germs} The twisted Boolean groupoid $(G,\Sigma)$ of the proposition is called the {\it twisted groupoid of germs} of $(\G,{\mathcal S})$.
 
\end{defn}
 
 \vskip4mm
 \subsection{Convolution algebras} The construction of the convolution algebra $C_c(G,\Sigma)$ of a twisted Boolean groupoid $(G,\Sigma)$ is a particular case of a construction valid for every twisted locally compact groupoid with a Haar system $\lambda$, which can be found in many places, in particular in \cite[Section 1]{ren:twisted extensions}. Since $G$ is \'etale, $\lambda^x$ and $\lambda_x$ are respectively the counting measures on $G^x=r^{-1}(x)$ and on $G_x=d^{-1}(x)$. When $G$ is Hausdorff, the elements of $C_c(G,\Sigma)$ are compactly supported continuous functions $f:\Sigma\to{\bf C}$ which are homogeneous in the sense that $f(\theta\sigma)=\overline\theta f(\sigma)$ for all $(\theta,\sigma)\in \t\times\Sigma$. When $G$ is not Hausdorff, they are finite sums of homogeneous compactly supported continuous functions functions defined on $\t$-invariant open Hausdorff subsets $U$ of $\Sigma$ and extended by 0 outside $U$. In all cases, they are finite sums of functions $f=\Delta_S(b\circ s)$, where $S$ is a compact continuous bisection of $\Sigma$, $b\in C_c(G^{(0)})$ and $\Delta_S(\sigma)=\overline\theta$ if $\sigma=\theta(Ss(\sigma))$, where $\theta\in\t$ and 0 if $\sigma$ does not belong to $\t S$. The convolution product and the involution are respectively given by
  \[f*g(\sigma)=\int_G f(\tau) g(\tau^{-1}\sigma) d\lambda^{r(\sigma)}(\tau'),\qquad
f^*(\sigma)=\overline{f(\sigma^{-1})}\]
where $\tau'=p(\tau)$ is the image in $G$ of $\tau\in\Sigma$. We warn the reader of an abuse of notation in the integral: the function $\tau\mapsto f(\tau) g(\tau^{-1}\sigma)$ should be replaced by the function it defines on the quotient $G=\Sigma/\t$. We can view $b\in C_c(G^{(0)})$ either as a multiplier of $C_c(G,\Sigma)$ by defining $fb(\sigma)=f(\sigma) b(d(\sigma))$ or as an element of $C_c(G,\Sigma)$ by defining $\tilde b(x,\theta)=\overline\theta b(x)$ and $\tilde b(\sigma)=0$ if $\sigma$ does not belong to $X\times\t$. A similar construction defines the Steinberg algebra $A_R(G,\Sigma)$ of $(G,\Sigma)$, where $R$ is a ring (see \cite[Section 3]{ste:isga} and \cite{acclmr, accclmrss} for the twisted version).

\begin{lem}\label{map u} Let $(G,\Sigma)$ be a twisted Boolean groupoid. For $S$ and $T$ in ${\rm Bis}(\Sigma))$, we have $\Delta_S*\Delta_T=\Delta_{ST}$, $\Delta_S^*=\Delta_{S^{-1}}$. If $S$ belongs to ${\mathcal T}(G^{(0)})$ and is given by a function $b$, then $\Delta_S=\tilde b$ as defined above.
\end{lem}

Therefore  $A(G,\Sigma)$ and $C_c(G,\Sigma)$ carry an injective representation of the twisted inverse semigroup $({\rm Bis}(G), {\rm Bis}(\Sigma))$ as a $*$-semigroup of partial isometries and span the algebras. In the next lemma, $\F (X)$ is the space of bounded complex-valued functions on $X$.

\begin{lem}\label{Q} Let $(\G,{\mathcal S})$  be a twisted Boolean inverse semigroup with twisted groupoid of germs $(G,\Sigma)$. The map $Q:C_c(G,\Sigma)\to \F(G^{(0)})$ defined by $Q(f)(x)=f(x,1)$ (where we view $G^{(0)}\times\t$ as a subgroupoid of $\Sigma$) satisfies
\begin{enumerate}
\item $Q$ is linear,
\item $Q(f^*)=Q(f)^*$ for all $f\in C_c(G,\Sigma)$,
\item $Q(f*f^*)(x)=\int |f|^2 d\lambda^x$ for all $f$ in $C_c(G,\Sigma)$ and all $x$ in $G^{(0)}$,
\item  $Q(\Delta_S b)=Q(\Delta_S)b$ for all $S\in{\mathcal S}$ and all $b$ in $C_c(G^{(0)})$,
\item $f([S,x])=Q(\Delta_S^**f)(x)$ for all $(S,x)$ in ${\mathcal S}*G^{(0)}$ and all $f$ in $C_c(G,\Sigma)$ and
\item $Q(C_c(G,\Sigma))=C_c(G^{(0)})$ if and only if $G$ is Hausdorff.  
\end{enumerate}
\end{lem}

\begin{proof}
The assertions (i) to (v) are easy calculations. Let us prove (vi). If $G$ is Hausdorff, the elements of $C_c(G,\Sigma)$ are continuous functions on $\Sigma$. Therefore, their restrictions to $G^{(0)}\times\{1\}$ are continuous. Suppose conversely that $Q(f)\in C_c(G^{(0)})$  for all $f\in C_c(G,\Sigma)$. This is true in particular for $f=\Delta_S$ where $S\in{\rm Bis}(\Sigma)$.  Then ${\bf 1}_{p(S)\cap G^{(0)}}=|Q(\Delta_S)|$ belongs to $C_c(G^{(0)})$. Since $C_c(G^{(0)})\subset C_c(G)$, $p(S)\cap G^{(0)}$ is a closed subset of $G$. This implies that $G^{(0)}$ is closed in $G$, hence $G$ is Hausdorff.
\end{proof}
\vskip 2mm
The $*$-algebra $C_c(G,\Sigma)$ can be completed with respect to various C*-norms, in particular the reduced norm and the full norm (see \cite{ren:rep, ren:twisted extensions}). According to \cite[Proposition 4.1]{ren:approach}, $|f(\sigma)|\le \|f\|_{\rm red}$ for all $f\in C_c(G,\Sigma)$ and all $\sigma\in\Sigma$. Therefore, $C^*_{\rm red}(G,\Sigma)$ embeds into the space of bounded complex-valued homogeneous functions on $\Sigma$. When we view the elements of $C^*_{\rm red}(G,\Sigma)$ as functions, it can be shown that the product of two elements is still given by the convolution product (see \cite[Proposition 2.21]{bfpr} or \cite[Corollary 5.4]{dwz} when $G$ is Hausdorff).

\begin{lem}\label{Q2} The above restriction map $Q$ can be extended to a bounded linear map from  $C^*_{\rm red}(G,\Sigma)$ into $\F (G^{(0)})$ endowed with the norm of uniform convergence such that
\begin{enumerate}
 \item $Q$ is positive;
 \item $Q$ is faithful;
 \item $Q(fb)=Q(f)b$ for all $(f,b)\in C^*_{\rm red}(G,\Sigma)\times C_0(G^{(0)})$.
 \end{enumerate}
\end{lem}

\begin{proof}  Since $|Q(f)(x)|\le \|f\|_{\rm red}$ for $f\in C_c(G,\Sigma)$ and $x\in G^{(0)}$, $Q$ extends by continuity to a bounded linear map from  $C^*_{\rm red}(G,\Sigma)$ into $\F (G^{(0)})$. It is shown as in the case when $G$ is Hausdorff (\cite[Proposition 4.3]{ren:cartan}) that $Q$ satisfies (i), (ii) and (iii).
\end{proof}

\begin{thm}\label{universal} Let $(\G,{\mathcal S})$ be a Boolean twisted inverse semigroup and $(G,\Sigma)$ its twisted groupoid of germs. Then the full C*-algebra $C^*(G,\Sigma)$ is universal with respect to the representations of $(\G,{\mathcal S})$. 
\end{thm}

\begin{proof} This is a corollary of \thmref{twisted rep} because the representations of $C_c(G,\Sigma)$ extend uniquely to representations of $C^*(G,\Sigma)$ (see \cite[Proposition 3.2]{qs:C*-actions}).
\end{proof}

This theorem has its origin in the pioneering work of Paterson (see \cite[Theorem 3.2.2]{pat:gpd}). Since then, various generalizations of Paterson's theorem have appeared, often with additional assumptions (for example \cite[Proposition 2.14]{be:fell} or \cite[Corollary 5.6]{bm:isg}).\\
We shall resume the study of the pair $(C_{\rm red}^*(G,\Sigma), C_0(G^{(0)}))$ in the next section after the introduction of the partial isometry normalizer.

\vskip4mm
\section{Abstract realization of a Cartan pair}

Our main example of a twisted Boolean inverse semigroup is the partial isometry normalizer of a Cartan subalgebra in a C*- or a W*-algebra. Of course, this is interesting only for C*-algebras containing many projections. Adapting \cite[Definition II.4.9]{ren:approach} and \cite[Definition 4.1]{ren:cartan}, we define

\begin{defn} Let $A$ be a C*-algebra and $B$ be an abelian C*-subalgebra.  We let $\B$ be the Boolean algebra of projections of $B$ and $X$ be the dual space of $\B$. 
\begin{enumerate}
 \item The {\it partial isometry normalizer}  ${\rm PIN}(B)$ of $B$ is the inverse semigroup
 of partial isometries $u$ of $A$ such that $d(u),r(u)\in B$, $uBu^*\subset B$ and $u^*Bu\subset B$.
 \item The semigroup of partial isometries of $B$ is denoted by ${\rm PI}(B)$. It coincides with the {\it trivial Clifford inverse semigroup} ${\mathcal T}(X)$ defined before \defnref{isg twist}.
 \item The action by conjugation on $B$ of a partial isometry $u\in{\rm PIN}(B)$ defines a homeomorphism  $\varphi_u: {\rm dom}(u)\to {\rm ran}(u)$, where ${\rm dom}(u)$ [resp. ${\rm ran}(u)$] is the compact open subset corresponding to the projection $d(u)$ [resp. $r(u)$]. The {\it ample pseudogroup} ${\mathcal G}(B)$ of $B$  is the pseudogroup  of these partial homeomorphisms of $X$.\end{enumerate}
\end{defn}

In \cite{ren:cartan}, the normalizer $N(B)$ is defined as the semigroup of all elements of $A$ which normalize $B$. The inverse semigroup ${\mathcal N}(B)$ of \cite[Definition II.4.9]{ren:approach} agrees with ${\rm PIN}(B)$ defined here. When $B$ contains an approximate unit for $A$, then the conditions  $d(u),r(u)\in B$ are implied by the others and one has 
\[{\rm PIN}(B)=\{u\in {\rm PI}(A): uBu^*\subset B\,{\rm and}\, u^*Bu\subset B\}.\]

 Moreover, if $B$ is maximal abelian and ${\rm PIN}(B)$ generates $A$, $B$ always contains an approximate unit for $A$ (\cite[Theorem 2.6]{pit:unit}). The pseudogroup $\G(B)$ is the quotient of ${\rm PIN}(B)$ by its centralizer, namely the set of elements which commute with every idempotent. Equivalently, it is the image of ${\rm PIN}(B)$ under the Munn representation (see \cite[Section 4]{law:isg3}). 
 
\begin{lem} Let $A$ be a C*-algebra and $B$ be an abelian C*-subalgebra. The inverse subsemigroup ${\rm PI}(B)$ of ${\rm PIN}(B)$ is normal, in the sense that
\begin{enumerate}
 \item ${\rm PI}(B)$ has the same idempotents as ${\rm PIN}(B)$;
 \item for all $u\in {\rm PI}(B)$ and $e\in\B$, $ue=eu$;
 \item for all $u\in{\rm PIN}(B)$ and $v\in {\rm PI}(B)$, $uvu^*$ belongs to ${\rm PI}(B)$.
\end{enumerate}
Therefore, the quotient ${\mathcal H}(B)\defequal{\rm PIN}(B)/{\rm PI}(B)$ for the congruence $u\sim v$ if and only if $d(u)=d(v)$ and $u^*v\in{\rm PI}(B)$ is an inverse semigroup. 
 \end{lem}

\begin{proof} The items (i), (ii) and (iii) are clear. The definition of the quotient by a normal inverse subsemigroup belongs to the general theory of inverse semigroups (see for example \cite{law:inverse,law:isg3}). 
\end{proof}

Thus $\G(B)$ is a quotient of ${\mathcal H}(B)$. We shall be mostly concerned by the case when these inverse semigroups are equal.

\begin{prop}\label{masa} Let $A$ be a C*-algebra and $B$ be a maximal abelian C*-subalgebra of $A$ generated by its projections.  Then ${\mathcal H}(B)=\G(B)$.
\end{prop}

\begin{proof} This is the Remark following \cite[Definition II.4.9]{ren:approach}.
\end{proof}

\begin{prop}\label{tbisg} Let $B$ be an abelian C*-subalgebra of a C*-algebra $A$. Then the pair $({\mathcal H}(B), {\rm PIN}(B))$ is a twisted Boolean inverse semigroup.
\end{prop}

\begin{proof} By definition,
\[{\rm PI}(B)\stackrel{i}{\rightarrowtail}{\rm PIN}(B)\stackrel{p}{\twoheadrightarrow}  {\mathcal H}(B)\]
 is an exact sequence of inverse semigroups. If $(u,v)$ is an orthogonal pair in ${\rm PIN}(B)$, $u+v$ belongs to  ${\rm PIN}(B)$ and is the join of $u$ and $v$ in the ordered set ${\rm PI}(A)$ of partial isometries of $A$, and a fortiori in ${\rm PIN}(B)$. One deduces that ${\mathcal H}(B)$ has the same property. Therefore, ${\rm PI}(B)$, ${\rm PIN}(B)$ and ${\mathcal H}(B)$ are Boolean inverse semigroups with $\B$ as common Boolean algebra of idempotents. Moreover ${\rm PI}(B)$ is the trivial Clifford inverse semigroup over the dual space of $\B$ and the action of ${\rm PIN}(B)$ on ${\rm PI}(B)$ is the trivial extension of its action on $\B$.
\end{proof}

Let $B$ be an abelian C*-subalgebra $B$ of a C*-algebra $A$. We can construct the twisted groupoid of germs $(H,\Sigma)$  of the Boolean inverse semigroup $({\mathcal H}(B),{\rm PIN}(B))$. We immediately have an injective map $j:{\rm PIN}(B)\to C_c(H,\Sigma)$, associating to $u\in{\rm PIN}(B)$ the function $\Delta_{S(u)}$, where $S(u)$ is the compact continuous bisection defined by $u$. From \lemref{map u}, the map $j$ satisfies $j(uv)=j(u)*j(v)$ for $u,v\in{\rm PIN}(B)$, $j(u^*)=j(u)^*$ for $u\in{\rm PIN}(B)$ and $j(b)=\hat b$, where $\hat b$ is the Gelfand transform of $b$ for $b\in{\rm PI}(B)$. In the following, we do not write the Gelfand transform to ease the formulas and write $B=C_0(X)$ and $B_c=C_c(X)$.

\begin{prop}\label{tilde pi}  Let $B$ be an abelian C*-subalgebra of a C*-algebra $A$. Assume that $B$ is generated by its projections. Let ${\rm PIN}(B)B_c$ be the subspace of finite sums of elements of the form $ub$ with $u\in{\rm PIN}(B)$ and $b\in B_c$.  Let $(H,\Sigma)$ be the twisted groupoid of germs of $({\mathcal H}(B),{\rm PIN}(B))$.Then
\begin{enumerate}
 \item ${\rm PIN}(B)B_c$ is a $*$-subalgebra of $A$.
 \item There is a unique $*$-homomorphism $\tilde\pi$ of $C^*(H,\Sigma)$ into $A$ sending $C_c(H,\Sigma)$ onto ${\rm PIN}(B)B_c$ and such that
\[\tilde\pi(\sum j(u_i)b_i)=\sum u_ib_i,\] 
where the sum is finite, $u_i \in {\rm PIN}(B)$, and $b_i \in B_c$.
\end{enumerate}
 \end{prop}
 
 \begin{proof}  We deduce (i) from the relation $ub=\alpha_u(b)u$, where $\alpha_u(b)=ubu^*$. Let us prove (ii). By construction, the twisted Boolean inverse semigroup $({\mathcal H}(B), {\rm PIN}(B))$  is represented as a twisted semigroup of partial isometries of $A$. By \thmref{universal}, this representation extends to a representation $\tilde\pi$ of the full C*-algebras $C^*(H,\Sigma)$. If $f\in C_c(H,\Sigma)$ is written  as a finite sum $\sum j(u_i) b_i$ as in the proposition, then $\tilde\pi(f)=\sum u_ib_i$. Note that $\tilde\pi(C_c(H,\Sigma))={\rm PIN}(B)B_c$. 
 \end{proof}
 
 Exel and Pitts give a similar result, namely \cite[Theorem 3.8.13]{ep:char}, with the sole assumption that $B$ is regular and where the twisted groupoid $(H,\Sigma)$ is constructed from a generating inverse subsemigroup of the normalizer $N(B)$.

\begin{prop}\label{tilde j}  Let $B$ be an abelian C*-subalgebra of a C*-algebra $A$. Assume that $B$ is generated by its projections. Let $(H,\Sigma)$ be the twisted groupoid of germs of $({\mathcal H}(B),{\rm PIN}(B))$. Then \tfae
\begin{enumerate}
\item There exists a linear map $P:{\rm PIN}(B)B_c\to \F(X)$ (the space of bounded complex-valued functions on $X$)
 such that
\begin{enumerate}
\item $P(u^*)=P(u)^*$ for all $u\in {\rm PIN(B)}$,
 \item $P(ub)=P(u)b$ for all $(u,b)\in {\rm PIN(B)}\times B_c$, and
 \item for $(u,x)\in{\rm PIN}(B)*X$, $[[u,x]]\not= x\,\Rightarrow\, P(u)(x)=0$, where the notation is that of \propref{isgtwist to gpdtwist}.
\end{enumerate}
 \item $\tilde\pi$ is a $*$-isomorphism of $C_c(H,\Sigma)$ onto ${\rm PIN}(B)B_c$.
\end{enumerate}
 \end{prop}
 
 \begin{proof} Suppose that $\tilde\pi$ is a $*$-isomorphism of $C_c(H,\Sigma)$ onto ${\rm PIN}(B)B_c$. Let  $P=Q\circ\tilde\j$ where $Q$ is defined in \lemref{Q} and $\tilde\j$ is the inverse of $\tilde\pi$. From \lemref{Q}, $P$ is a linear map from  ${\rm PIN}(B)B_c$ to $\F(X)$ which satisfies (a) and (b).  By definition, $\tilde\j(u)=j(u)=\Delta_{S(u)}$ for $u\in{\rm PIN}(B)$. Therefore $|P(u)|={\bf 1}_{S'(u)\cap H^{(0)}}$, where $S'(u)$ is the bisection of $H$ defined by $p(u)\in{\mathcal H}$. We conclude that $P(u)(x)=0$ if $[[u,x]]\notin H^{(0)}$.
 \vskip2mm 
 Suppose conversely the existence of $P$ as in (i). For $a\in{\rm PIN}(B)B_c$, and $(v,x)\in {\rm PIN}(B)*X$, we define $\hat a(v,x)=P(v^*a)(x)$. If $b\in{\rm PI}(B)$ satisfies $b(x)=1$, then $\hat a(vb,x)=\hat a(v,x)$. Therefore, $\hat a(v,x)$ depends only on the equivalence class $[v,x]\in\Sigma$. The function $\hat a$ is homogeneous: let $\theta\in\t$ and $t\in{\rm PI}(B)$ such that $t(x)=\theta$. Then 
 \[\hat a((v,x) \theta)=\hat a(vt,x)=P((vt)^*a)(x)=P(t^*v^*a)(x)=t^*(x)P(v^*a)(x)=\overline\theta\hat a(v,x).\]  
 Let us show that, for $u\in{\rm PIN}(B)$, $\hat u=\Delta_{S(u)}$. Let $(v,x)\in {\rm PIN}(B)*X$. We have seen that $\hat u(v,x)=P(v^*u)(x)$ is non-zero if and only if $v^*u$ belongs to ${\rm PI}(B)$ and $(v^*u)(x)$ is non-zero. Then $\hat u(v,x)=P(v^*u)(x)=(v^*u)(x)$. This agrees with the definition of $\Delta_{S(u)}$. For $a=ub$ where $u\in{\rm PIN}(B)$ and $b\in B_c$, we have
\[\hat a(v,x)=P(v^*ub)(x)=P(v^*u)(x) b(x)=\hat u(v,x) b(x).\]
Thus $\hat a=j(u)b$. Therefore the map $\tilde\j$ sending $a\in {\rm PIN}(B)B_c$ to $\hat a\in C_c(H,\Sigma)$ is an inverse of the restriction of $\tilde\pi$ to $C_c(H,\Sigma)$. 
\end{proof}

We give a name to condition (c), which is similar to condition (4) of \cite[Theorem 3.1]{abcclmr}.

\begin{defn} A map $P:{\rm PIN}(B)\to \F(X)$ is called {\it separating} if for $u\in{\rm PIN}(B)$ and $x\in{\rm dom}(u)$ such that the germ of $\varphi_u$ at $x$  is not trivial, we have $P(u)(x)=0$.
 \end{defn}

\begin{cor}\label{char} Let $B$ an abelian C*-subalgebra a C*-algebra  $A$. Assume that $B$ is generated by its projections and that $A$ is generated by ${\rm PIN}(B)$. Let $(H,\Sigma)$ be the twisted Boolean groupoid corresponding to $({\mathcal H}(B), {\rm PIN}(B))$. Then \tfae
\begin{enumerate}
\item The map $\tilde\pi$ of \propref{tilde pi} gives an isomorphism of $C^*_{\rm red}(H,\Sigma)$ onto $A$.
\item There exists a linear map $P:A\to\F(X)$, where $X$ is the spectrum of $B$, such that
\begin{enumerate}
 \item $P$ is positive;
 \item $P$ is faithful;
 \item $P(ub)=P(u)b$ for all $(u,b)\in {\rm PIN(B)}\times B$;
 \item the restriction of $P$ to ${\rm PIN}(B)$ is separating.
\end{enumerate}
 \end{enumerate}
 \end{cor}

\begin{proof}  Assume (i). We define $P=Q\circ\tilde\pi^{-1}$, where $Q$ is as in \lemref{Q2}. Since $Q$ satisfies the properties (a) to (c), so does $P$. According to \propref{tilde j}, it also satisfies (d) (we recall that by assumption ${\rm PIN}(B)={\rm Bis}(\Sigma)$). Assume (ii). The restriction of $P$ to ${\rm PIN}(B)B_c$ satisfies the conditions of \propref{tilde j}. Therefore, $\tilde\pi$ is a $*$-isomorphism of $C_c(H,\Sigma)$ onto ${\rm PIN}(B)B_c$. The pairs $(C_c(H,\Sigma), Q)$ and $({\rm PIN}(B)B_c, P)$ are isomorphic, where we view $Q$ and $P$ as completely positive linear maps into a common abelian C*-algebra $\widetilde B$. The KSGNS construction provides isomorphic $\widetilde B$-Hilbert modules $L^2(Q)$ and $L^2(P)$ and conjugate representations $\rho_Q$ and $\rho_P$ of the $*$-algebras $C_c(H,\Sigma)$ and ${\rm PIN}(B)B_c$ on these Hilbert modules. By construction (cf. \cite[Theorem 2.10]{ks:regular}), the reduced C*-algebra $C^*_{\rm red}(H,\Sigma)$ is the closure of $\rho_Q(C_c(H,\Sigma))$. The representation $\rho_P$ extends to $A$ by density of ${\rm PIN}(B)B_c$. Since $P$ is faithful, $\rho_P$ is isometric. Thus we have the equality of the norms $\|f\|_{\rm red}$ and $\|\tilde\pi(f)\|$ for $f\in C_c(H,\Sigma)$. Therefore $\tilde\pi$ extends to an isomorphism of $C^*_{\rm red}(H,\Sigma)$ onto $A$.\end{proof}
Note that, by construction, we have the equality ${\rm PIN}(B)={\rm Bis}(\Sigma)$. We shall prove in \propref{char2} that this equality implies that $B$ is maximal abelian. Thus, if these equivalent conditions are satisfied, $H$ is the Weyl groupoid, i.e. the groupoid of germs of the ample pseudogroup $\G(B)$. \corref{char} is an algebraic characterization of the C*-algebras which can be expressed as the reduced C*-algebra of a twisted Boolean groupoid satisfying the local bisection hypothesis of \cite[Definition 4.1]{accclmrss}. It is a mild extension to non-Hausdorff Boolean groupoids of \cite[Theorem 5.9]{ren:cartan}.

\vskip2mm
Let us complete our study of the inclusion of $B=C_0(G^{(0)})$ into $A=C^*_{\rm red}(G,\Sigma)$, where $(G,\Sigma)$ is a twisted Boolean groupoid.

\begin{lem} Let $(G,\Sigma)$ be a twisted Boolean groupoid.
\begin{enumerate}
 \item if $S$ is a compact continuous bisection of $\Sigma$, then $\Delta_S$ is a partial isometry normalizer of $B$ in $A$;
 \item the map $\Delta: S\mapsto \Delta_S$ is an injective Boolean inverse semigroup morphism of ${\rm Bis}(\Sigma)$ into ${\rm PIN}(B)$;
 \item if $G$ is effective, a partial isometry normalizer $u\in{\rm PIN}(B)$ is in the image of $\Delta$ if and only if its strict support ${\rm supp}'(u)=\{\gamma\in G: |u|(\gamma)\not=0\}$ is an open subset of $G$.
\end{enumerate}
 \end{lem}

\begin{proof}
 (i) It results from \lemref{map u} that $\Delta_S$ is a partial isometry with domain $d(S)$ and range $r(S)$ in $B$. By \cite[Proposition 4.7.(i)]{ren:cartan}, it normalizes $B$.\\
 \lemref{map u} shows that $\Delta$ is a morphism of $*$-semigroup. Its restriction to the Boolean algebra $\B$ of compact open subsets of $G^{(0)}$ is the usual embedding of $\B$ into the Boolean algebra of projections of $B$. If $\Delta_S$ belongs to $\B$, then $S$ is contained in $G^{(0)}$. This gives (ii). The proof of (iii)  is essentially that of \cite[Proposition 4.7]{ren:cartan} which we reproduce here. Let $u\in{\rm PIN}(B)$ with open strict support.  As earlier $\varphi_u: {\rm dom}(u)\to {\rm ran}(u)$ denotes the partial homeomorphism of $G^{(0)}$ induced by the action of $u$ on $B$ by conjugation. We fix $x\in {\rm dom}(u)$. The equality
 \[b(\varphi_u(x))=\int |u(\tau)|^2 b\circ r(\tau) d\lambda_x(\tau')\] 
 holds for all $b\in B$. In other words, the pure state $\delta_{\varphi_u(x)}$ is expressed as a (possibly infinite) convex combination of pure states of $B$. This implies that $u(\tau)=0$ if $r(\tau)\not=\varphi_u(x)$. Denoting $S'=p(S)$ the image of $S$ in $G$, we have shown $S'\subset T$ where
 \[T=\{\gamma\in G: d(\gamma)\in{\rm dom}(u)\quad{\rm and}\quad r(\gamma)=\varphi_u\circ s(\gamma)\}.\]
 We deduce the inclusions $S'{S'}^{-1}\subset TT^{-1}\subset G'$. Since $S'{S'}^{-1}$ is open and $G$ is effective, it must be contained in $G^{(0)}$. By the same token, ${S'}^{-1}S'$ is also contained in $G^{(0)}$. This show that $S'$ is a bisection. It is open by assumption and also compact because homeomorphic to ${\rm dom}(u)$. To conclude, we pick $S\in{\rm Bis}(\Sigma)$ such that $p(S)=S'$. Then the strict support of $\Delta_{S^{-1}}*u$ is ${S'}^{-1}S'$. Therefore, there exists $b\in{\rm PI}(B)$ such that $u=\Delta_S b=\Delta_{Sb}$. Conversely, if $S$ is a continuous compact bisection of $\Sigma$, the strict support of $\Delta_S$, which is the image $p(S)$ of $S$ in $G$, is a compact open subset of $G$.
\end{proof}

 Along the lines of  \cite[Theorem 3.1]{abcclmr} which  gives a similar result in the case of a locally compact Hausdorff \'etale groupoid and where the normalizer $N(B)$ is used instead of the partial isometry normalizer ${\rm PIN}(B)$, we have:

\begin{prop}\label{char2} (cf. \cite[Theorem 3.1]{abcclmr}) Let $(G,\Sigma)$ be a twisted Boolean groupoid. Consider the inclusion of $B=C_0(G^{(0)})$ into $A=C^*_{\rm red}(G,\Sigma)$ and the four conditions:
\begin{enumerate}
 \item ${\rm PIN}(B)={\rm Bis}(\Sigma)$;  
 \item the restriction map $Q$ of \lemref{Q2} is separating;
  \item $B$ is maximal abelian;
  \item $G$ is effective.
\end{enumerate}
Then, the conditions (i) and (ii) are equivalent and imply (iii), which in turn implies (iv). If $G$ is Hausdorff, these four conditions are equivalent.
 \end{prop}

\begin{proof} Let $(\G,{\mathcal S})$ be the twisted ample semigroup of $(G,\Sigma)$. It is represented in $A$ through the map $S\mapsto \Delta_S$. The kernel ${\mathcal T}(G^{(0)})$ of the twist  is sent onto ${\rm PI}(B)$. \\
(i)$\Rightarrow$(ii) Let $u\in {\rm PIN}(B)$. By assumption, there exists $S\in{\mathcal S}$ such that $u=\Delta_S$. If $Q(u)(x)\not=0$, then $x$ belongs to the open subset $p(S)\cap r(S)$ of $G^{(0)}$. There exists a compact open set $V$ containing $x$ and contained in $p(S)\cap r(S)$. Since $p(S_V)$ is contained in $G^{(0)}$, the germ of $\varphi_S$ at $x$ is trivial.\\
  (ii)$\Rightarrow$(i) In the following, the product in $A$ is written $ab$.  Let $u\in {\rm PIN}(B)$ and $\sigma\in\Sigma$ such that $u(\sigma)\not=0$. There exists $S\in {\mathcal S}$ containing $\sigma$. Then, $x=r(\sigma)$ satisfies
  \[Q(u\Delta_{S^{-1}})(x)=u\Delta_{S^{-1}}(x)=u(\sigma)\not=0.\]
  Therefore, there exists $b\in {\rm PI}(B)$ such that $b(x)\not=0$ and $(u\Delta_{S^{-1}})|b|=b$.
  Replacing $v=\Delta_S$ by $|b|v$, we obtain  $uv^*=b$ and $ud(v)=bv$, which can be written $ud(w)=w$, where $w=bv$. We have shown that for all $\sigma$ such that $u(\sigma)\not=0$, there exists $S\in{\mathcal S}$ such that $d(\sigma)\in d(S)$ and $ud(S)=\Delta_S$.  The family $(d(S))$ is an open cover of the support of the projection $u^*u$ in $G^{(0)}$. We extract a finite cover $(d(S_i))$. The family $(S_i)$ of the Boolean inverse semigroup ${\mathcal S}$ is compatible in the sense that, for all $i,j$, $S_i^{-1}S_j$ is an idempotent. Since $\mathcal S$ is a Boolean inverse semigroup, there exists $S=\vee_iS_i$ in $\mathcal S$ such that $Sd(S_i)=S_i$ for all $i$. Then $u=\Delta_S$.\\
  (i)$\Rightarrow$(iii) We assume that $B$ is strictly contained in its commutant $B'$; we will show the existence of $u$ in ${\rm PIN}(B)$ which is not in the image of $\Delta$. There exists $a\in B'\setminus B$ such that $a=a^*$ and $\|a\|\le 1$. Since $B$ admits an approximate unit of projections, there exists a projection $e\in B$ such that $ae\not\in B$. Then
$$u=ae+i\sqrt{e-a^2e}$$
satisfies $u^*u=uu^*=e$ and belongs to $B'$. Therefore, it belongs to ${\rm PIN}(B)$. Since $ae=\displaystyle{u+u^*\over 2}$, $u\notin B$. If $u$ is not of the form $\Delta_S$ for some $S\in{\rm Bis}(\Sigma)$, we are done. Therefore, we assume that $u=\Delta_S$ where $S\in{\rm Bis}(\Sigma)$. Since $u\in B'$, then according to \cite[Proposition II.4.7.(i)]{ren:approach}, $p(S)$ is contained in the isotropy group bundle $G'$ and even in its interior  because it is an open subset of $G$. We then consider the open subgroupoid $H=\bigcup_{k\in\Z}p(S)^k$, where $p(S)^0=Y$ is the support of $e$, endowed with the restriction of the twist $\Lambda=\Sigma_{|H}$. It is known (see \cite[Lemma 3.4]{bl:cartan II}) that $C^*_{\rm red}(H,\Lambda)$ embeds isometrically in $C^*_{\rm red}(G,\Sigma)$. This reduces the problem to the twisted group bundle $(H,\Lambda)$: indeed, a partial isometry of $C^*_{\rm red}(H,\Lambda)$ which normalizes $Be$ is an element of ${\rm PIN}(B)$; if it were given by a compact continuous bisection of $\Sigma$, this bisection would lie in $\Lambda$. Since $\Lambda$ is an abelian group bundle over $Y$,  $C^*_{\rm red}(H,\Lambda)$ is commutative and isomorphic to the C*-algebra $C(Z)$ of continuous functions on its spectrum $Z$. The embedding of $C(Y)$ into $C(Z)$ gives a continuous surjective map $q:Z\to Y$. Since $u\notin B$, there exists $y\in Y$ such that $H(y)\not=\{y\}$. The map sending an element of $C_c(H,\Lambda)$ to its restriction  to $\Lambda(y)$ extends to a $*$-homomorphism from $C^*_{\rm red}(H,\Lambda)$ onto $C^*_{\rm red}(H(y),\Lambda(y))$. Via the Gelfand transform, this map restricts $f\in C(Z)$ to $q^{-1}(y)$. Since the group $\Lambda(y)$ is abelian, the twist $(H(y),\Lambda(y))$ is trivial (\cite[Lemma 6.3]{ren:twisted extensions}). A trivialization gives an isomorphism of $C^*_{\rm red}(H(y),\Lambda(y))$ onto  $C^*_{\rm red}(H(y))$. It is not difficult to see via the Gelfand transform that the C*-algebra of a non-trivial cyclic group $\Z$ or $\Z/p\Z$ contains unitaries which are not multiples of the canonical unitaries. One then obtains unitaries in $C_{\rm red}^*(H(y),\Lambda(y))$ which are not of the form $\Delta_{\{\sigma\}}$ for some $\sigma\in\Lambda(y)$. We view such a unitary as a continuous complex-valued function $f$ on $q^{-1}(y)$ of modulus one. By the Tietze extension theorem, we can extend this function to a continuous function $g$ of modulus one on $Z$. This gives a unitary in $C^*_{\rm red}(H,\Lambda)$ which is not of the form $\Delta_T$ for some $T\in{\rm Bis}(\Lambda)$.\\
 (iii)$\Rightarrow$(iv) A bisection $S\in{\mathcal S}$ is a centralizer if and only if $\Delta_S$ belongs to the commutant $B'$ of $B$. Thus, if $B'=B$, this means that $\Delta_S$ belongs to ${\rm PI}(B)$ or, equivalently, that $S$ belongs to ${\mathcal T}(G^{(0)})$. Therefore $\G$ is fundamental (this means that $\G$ is a pseudogroup of partial homeomorphisms) and $G$ is effective.\\
 Suppose now that $G$ is Hausdorff. The equivalence of (iii) and (iv) is essentially established in \cite[Proposition II.4.7 (ii)]{ren:approach}. The proof that (iii) implies (ii) is identical to the proof of \cite[Lemma II.5.2]{ren:approach}.
  \end{proof}
  
In \cite{exe:nh}, Exel give examples of effective non-Hausdorff \'etale groupoids which do not satisfy (i) nor (iii). On the other hand, there exist non-Hausdorff Boolean groupoids which satisfy (iii). I do not know if the implication (i) $\Rightarrow$ (iii) is strict.\\

\vskip2mm
 We turn now to the context of von Neumann algebras. When $B$ is a commutative W*-algebra, or more generally a commutative AW*-algebra, its Boolean algebra of projections is complete. Correspondingly, its spectrum is a Stonian space, as defined below.

\begin{defn} A {\it Stonian space} is a topological space which is compact, Hausdorff and extremally disconnected.  \end{defn}
 
 \begin{defn} A {\it Stonian groupoid} is an \'etale groupoid whose unit space is a Stonian space. We shall say that a twisted groupoid $(G,\Sigma)$ is Stonian if $G$ is Stonian.
 \end{defn}

\begin{prop}\label{principal and Stonian}  Let $B$ be an abelian AW*-subalgebra of an AW*-algebra $A$. Then the groupoid of germs of the pseudogroup $\G(B)$ is principal and Stonian.
\end{prop}

\begin{proof} It is known that an abelian C*-algebra $B$ is an AW*-algebra if and only if its spectrum is a Stonian space. According to \cite[Proposition 2.1]{kklru:essential}, the isotropy bundle of an \'etale groupoid whose unit space is a Stonian space is open. Since the groupoid of germs of a pseudogroup of partial homeomorphisms is always effective, meaning that the interior of its isotropy is reduced to its unit space, its isotropy bundle is reduced to its unit space, which is the definition of a groupoid being principal.
\end{proof}

\begin{defn} A {\it complete Boolean inverse semigroup} $\G$ is an inverse semigroup such that
\begin{enumerate}
 \item $\G^{(0)}$ is a complete Boolean algebra and
 \item every orthogonal family of $\G$ admits a join.
\end{enumerate}
 \end{defn}

\begin{prop}  Let $B$ be an abelian W*-subalgebra of a W*-algebra $A$.  The pseudogroup $\G(B)$ is a complete Boolean inverse semigroup.
 \end{prop}

\begin{proof} If $(u_i)$ is an orthogonal family in ${\rm PIN}(B)$, $\sum_iu_i$, defined as a weak limit in $A$, exists, belongs to ${\rm PIN}(B)$ and is the join of the family. If $(\varphi_i)$ is an orthogonal family in $\G(B)$, we choose for all $i$ an element $u_i$ in ${\rm PIN}(B)$ which implements $\varphi_i$. The family $(u_i)$ is orthogonal and $p(\sum_iu_i)$ is the join of the family $(\varphi_i)$.
\end{proof}

\begin{rem} This result is still true when  $B$ be an abelian AW*-subalgebra of an AW*-algebra $A$. One can prove it directly by using
 \cite[Ch4, section 20, Theorem 1 (iii)]{ber:AW*} or by using the equivalence between complete Boolean inverse semigroups and Stonian groupoids.
\end{rem}

\begin{defn} When $B$ be an abelian W*-subalgebra of a W*-algebra $A$, the pseudogroup $\G(B)$ is called the full pseudogroup of the inclusion $B\subset A$.
\end{defn}

\vskip2mm

Let us recall the definition of a Cartan subalgebra of a W*-algebra (\cite[Definition 3.1]{fm:relations II}).

\begin{defn}\label{W*-Cartan}  A commutative W*-subalgebra $B$ of a W*-algebra $A$ is called a {\it Cartan subalgebra} of $A$ if
\begin{enumerate}
 \item $B$ is maximal abelian,
 \item ${\rm PIN}(B)$ generates $A$ as a W*-algebra,
 \item there exists a faithful normal conditional expectation $P$ of $A$ onto $B$. 
\end{enumerate}
We also say that $(A,B)$ is a W*-{\it Cartan pair}.
\end{defn}

The assumption of regularity made in \cite[Definition 3.1]{fm:relations II} is formally stronger since it involves the unitary normalizer instead of the partial isometry normalizer. Combined with \cite[Theorem 1]{fm:relations I}, \corref{fm} gives the equivalence of the two definitions. A direct proof of the equivalence can be found in \cite[pages 479-480]{cpz:bimodules}. Since $B$ is maximal abelian and contains the unit of $A$, we introduce the twisted Boolean inverse semigroup $(\G(B), {\rm PIN}(B))$ and its twisted groupoid of germs $(G(B),\Sigma(B))$. Note that according to \propref{principal and Stonian}, $G(B)$ is a principal groupoid.

\begin{defn} The twisted Boolean groupoid $(G(B),\Sigma(B))$ is called the {\it Weyl twist} of the W*-Cartan pair $(A,B)$. The inverse semigroup $\G(B)$ is called the {\it Weyl pseudogroup} of the Cartan pair.
 \end{defn}

\vskip2mm
We recall the construction of the W*-algebra of a twisted measure groupoid. We first need the notion of quasi-invariant measure.
In  \cite[I.3]{ren:approach}, a quasi-invariant measure for a locally compact groupoid $G$ with Haar system $\lambda$ is defined as a measure $\mu$ on $G^{(0)}$ such that the measure $\mu\circ \lambda$ on $G$ is equivalent to its inverse. For a Boolean groupoid one has the equivalent definition of a quasi-invariant measure (\cite[Proposition 3.2.2]{pat:gpd}):

\begin{defn} Let $G$ be a Boolean groupoid. A measure $\mu$ on $G^{(0)}$ is {\it quasi-invariant} if and only if for every compact open bisection $S$, the measures $(r_{|S})^*\mu$ and $(d_{|S})^*\mu$ are equivalent. A Boolean groupoid together with a quasi-invariant measure is called a {\it Boolean measure groupoid}.
\end{defn}

As in \cite[Section 4]{hah:regular}, the W*-algebra of a Boolean measure groupoid is defined as the von Neumann algebra of its regular representation:

\begin{defn}\label{vN} Let $(G,\Sigma)$ be a twisted Boolean groupoid and $\mu$ be a quasi-invariant probablity measure on $G^{(0)}$. We define $W^*(G,\Sigma,\mu)$ as the weak closure of ${\rm Reg}(C_c(G,\Sigma))$, where ${\rm Reg}$ is the GNS representation built on the state $\mu\circ Q$ of $C_c(G,\Sigma)$ and $Q$ is the conditional expectation $Q$ defined in \lemref{Q}.
\end{defn}

The Hilbert space of ${\rm Reg}$ is the completion $L^2(G,\Sigma, d^*\mu)$ of $C_c(G,\Sigma)$ for the scalar product
\[(f\,|\,g)=\mu\circ Q(f^**g)=\int \overline f g\, d(\mu\circ\tilde\lambda),\]
where $d^*\mu=\mu\circ\tilde\lambda$ and $\tilde\lambda$ is the system of counting measures on the fibres of $d:G\to G^{(0)}$. Then ${\rm Reg}(f)\xi=f*\xi$ for $f,\xi\in C_c(G,\Sigma)$.
\vskip2mm
A Cartan subalgebra $B$ of a W*-algebra $A$ is itself an abelian W*-algebra. We assume that it is countably decomposable and we fix a faithful normal state $\m$. This normal state extends to a probability measure $\mu$ on the spectrum of $B$. Since ${\rm PIN}(B)$ acts on $B$ by partial automorphisms, $\mu$ is quasi-invariant under $\G(B)$.
\vskip2mm
The following theorem which shows that every W*-Cartan pair is isomorphic to the W*-Cartan pair constructed from its Weyl twist is essentially the same as the abstract version of the Feldman-Moore theorem given in \cite{dfp:cartan}. An equivalent result, Theorem 12.3.8, is established in \cite{adp:II1} in the case of tracial von Neumann algebras.

\begin{thm}\label{abstract}
 Let $B$ be a Cartan subalgebra of a W*-algebra $A$ together with a faithful normal state $\m$ on $B$. Let $(G(B),\Sigma(B))$ be the Weyl twist of $(A,B)$. Let $\mu$ be the measure on the spectrum of $B$ extending $\m$. Then the $*$-isomorphism 
 \[\tilde\j: {\rm PIN}(B)B\to C_c(G(B),\Sigma(B))\] of \propref{tilde pi} exists and extends to an isomorphism 
 \[\Phi: A\to W^*(G(B),\Sigma(B),\mu)\] sending $B$ onto $C(G^{(0)})$.
\end{thm}

\begin{proof} The restriction to ${\rm PIN}(B)B$ of the conditional expectation $P$ satisfies the conditions of \propref{tilde j}: the conditions (a) and (b) are clear and \cite[Lemma 5.2]{ren:cartan} gives (c). Therefore, the inverse semigroup isomorphism
\[j:{\rm PIN}(B)\to{\rm Bis}(\Sigma(B))\]
 extends to a $*$-isomorphism 
 \[\tilde\j:{\rm PIN}(B)B\to C_c(G(B),\Sigma(B)).\]
 Since the states  $\omega_1={\m}\circ P$ of ${\rm PIN}(B)B$ and $\omega_2=\mu\circ Q$ of $C_c(G(B),\Sigma(B))$ match, they give isomorphic GNS representations $(H_1,\pi_1)$ and $(H_2,\pi_2)$: there exists a unitary operator $U: H_1\to H_2$ extending $\tilde\j$ such that $\pi_2=U\pi_1 U^{-1}$. By definition, $\pi_2(C_c(G(B),\Sigma(B))''$ is the von Neumann algebra $W^*(G(B),\Sigma(B),\mu)$ of the twisted measure groupoid $(G(B),\Sigma(B),\mu)$. On the other hand, since $\omega_1$ is normal, the $*$-homorphism 
\[\pi_1:{\rm PIN}(B)B\to\B(H_1)\] is also normal. Since the $*$-algebra ${\rm PIN}(B)B$ is $\sigma$-weakly dense in $A$, $\pi_1$ extends to a normal $*$-homorphism \[\pi:A\to\B(H_1).\] It is clear that $\pi$ is the GNS representation defined by the normal state ${\m}\circ P$ of $A$. Since $P$ is faithful, so is ${\m}\circ P$. Therefore $\pi$ is an isomorphism onto its image, which is $\pi_1({\rm PIN}(B)B)''$. Thus we have $U\pi(A)U^{-1}=W^*(G(B),\Sigma(B),\mu)$. We define the isomorphism $\Phi$ by $\Phi(a)=U\pi(a)U^{-1}$. If $b\in B$, $\pi(b)=\pi_1(b)$ is carried into $\pi_2(\hat b)$ which is the operator of multiplication $\xi\mapsto (\hat b\circ r)\xi$ on $L^2(G,\Sigma, d^*\mu)$.
\end{proof}

\begin{rem}
 A similar proof gives a similar result for an AW*-Cartan pair $(A,B)$. Instead of the GNS representation defined by the state ${\m}\circ P$ of $A$, one uses the AW*-version of the GNS construction given by H.~Widom in \cite{wid:embedding}: the conditional expectation $P:A\to B$ defines an AW*-module over $B$ and an isomorphic representation of $A$ as an AW*-subalgebra of the AW*-algebra of bounded operators of this module.
\end{rem}

\vskip4mm
\section{Concrete realization of a Cartan pair}

The main interest of the abstract realization of a Cartan pair given in the previous section, besides  being intrinsic, is to reduce the study of a W*-Cartan pair to that of a twisted measure Boolean inverse semigroup (or groupoid). In fact, we may forget that our twisted measure Boolean inverse semigroup $(\G,{\mathcal S},\m)$ arises from a Cartan pair (this is the case if and only if $\G$ is fundamental, i.e. $\G$ is a pseudogroup of transformations). We shall assume that $(\G,{\mathcal S},\m)$ is a twisted measure inverse semigroup as in the next definition.

\begin{defn} A {\it twisted measure inverse semigroup} is a twisted complete Boolean inverse semigroup $(\G,{\mathcal S})$ such that $\G^{(0)}$ has a unit together with a faithful measure ${\m}:\G^{(0)}\to [0,1]$, where {\it faithful} means that $\m$ is strictly positive on non-zero elements.
\end{defn}

The corresponding object through the non-commutative Stone equivalence recalled in the Appendix A is a twisted Stonian measure groupoid $(G,\Sigma,\mu)$:

\begin{defn} A twisted measure groupoid  $(G,\Sigma,\mu)$  is called {\it Stonian} if
\begin{enumerate}
 \item $(G,\Sigma)$ is a twisted Boolean groupoid;
 \item $G^{(0)}$ is a Stonian space and
 \item $\mu$ is a probability measure on $G^{(0)}$ which is strictly positive on non-empty open subsets and vanishes on nowhere dense Borel subsets. 
\end{enumerate}
\end{defn}

Just as measure algebras can be realized under appropriate separability assumptions by standard measure spaces, W*-algebras of twisted Stonian measure groupoids can be realized as W*-algebras of second countable twisted Boolean measure groupoids when they have a separable predual. We proceed in the same way as in the classical case of a measure algebra (see \cite[Section 40]{hal:measure} or \cite[Section 2.1]{gla:joinings}). We first consider the untwisted case. To avoid cluttering the proof too much, we isolate the following easy lemma.

\begin{lem}\label{idxpi} Assume that $\pi:X\to Y$ and $\rho: Z\to Y$ are surjective Borel maps, where $X,Y,Z$ are Borel spaces,  and that $(\alpha_y)$  is a Borel system of measures along the fibres of $\rho$. Let $\mu$ be a probability measure on $X$ such that for every Borel subset $A$ of $X$, there exists a Borel subset $B$ of $Y$ such that $\mu(A\Delta\pi^{-1}(B))=0$. Let $\nu=\int\alpha_{\pi(x)} d\mu(x)$ be the integrated measure on the fibred product $Z\times_Y X$. Then for every Borel subset $C$ of $Z\times_Y X$ of finite measure, there exists a Borel subset $D$ of $Z$ such that $\nu(C\Delta\pi_1^{-1}(D))=0$, where $\pi_1$ is the projection of $Z\times_Y X$ onto $Z$.
\end{lem}

\begin{proof} Given $E\subset Z$ and $A\subset X$, we write $E*A=\{(z,x)\in E\times A: \rho(z)=\pi(x)\}$. Since the Borel structure of $Z\times_Y X$ is deduced from the product Borel structure of $Z\times X$, it suffices to prove the result for $C=E*A$, where $A$ is a Borel subset of $X$ and $E$ is a Borel subset of $Z$ such that $\nu(E*A)$ is finite. Let $B$ be a Borel subset of $Y$ such that $\mu(A\Delta\pi^{-1}(B))=0$. Let $D=E\cap \rho^{-1}(B)$. Then $\pi_1^{-1}(D)=E*\pi^{-1}(B)$. Therefore,
\[C\Delta\pi_1^{-1}(D)=(E*A)\Delta (E*\pi^{-1}(B))=E*(A\Delta\pi^{-1}(B)).\]
Then \[\nu(C\Delta\pi_1^{-1}(D))=\int_{A\Delta\pi^{-1}(B)}\alpha_{\pi(x)}(E) d\mu(x) =0.\]
\end{proof}

Mutiplier actions are defined in the Appendix A (\defnref{multiplier action} and paragraph following \defnref{proper anchor}). Note that we assume that $G^{(0)}$ is compact; therefore the anchor of the action is proper. We recall that $d^*\mu=\mu\circ\tilde\lambda$, where $\tilde\lambda$ is the system of counting measures on the fibres of the domaine map. One defines similarly $r^*\mu=\mu\circ\lambda$.

\begin{prop}\label{main} Let $(G,\mu)$ be a Stonian measure groupoid such that the Hilbert space $L^2(G,d^*\mu)$ is separable. Then there exist a second countable Boolean groupoid $H$ acting by multipliers on $G$ with anchor map $\pi: G^{(0)}\to H^{(0)}$ and a surjective isometry $J: L^2(H, d^*(\pi_*\mu))\to L^2(G, d^*\mu)$ extending the $*$-homomorphism $\tilde\j: C_c(H)\to C_c(G)$ defined by the multiplier action. We can furthermore assume that the ample semigroup of $H$ is dense in the ample semigroup of $G$ with respect to the metric $\rho(S,T)=\nu(S\Delta T)$ associated with the measure $\nu=d^*\mu+r^*\mu$.
\end{prop}

\begin{proof} Let \(B(d^*\mu)=\B orel(G)/{\mathcal N}(d^*\mu)\) be the measure algebra of the measure space $(G,d^*\mu)$. The measure $d^*\mu$ being possibly infinite, we consider the set $B_{\rm finite}(d^*\mu)$ of elements of $B(d^*\mu)$ of finite measure. Since the Hilbert space $L^2(G,d^*\mu)$ is separable, according to \cite[Section 42 (1)]{hal:measure}, the metric space $(B_{\rm finite}(d^*\mu), \rho_d)$, where $\rho_d(E,F)=d^*\mu(E\Delta F)$, is separable. We can view $\G={\rm Bis}(G)$ as a subspace of $B_{\rm finite}(d^*\mu)$: if $S\in\G$, $d^*\mu(S)=\mu(d(S))$ is finite and if $S,T\in\G$ satisfy $d^*\mu(S\Delta T)=0$, then $S=T$ because $\mu$ is strictly positive on non-empty open sets. Therefore $\G$ is separable with respect to the metric $\rho_d$. We pick a countable dense subset of $\G$. The Boolean inverse subsemigroup $\mathcal H$ generated by this subset is countable. Therefore (see \propref{properties}(iii)), its groupoid of germs $H$ is second countable. According to \thmref{equivalence}, the inclusion $j:{\mathcal H}\to \G$ gives a multiplier action of $H$ on $G$. Let  $\pi:X\to Y$ be the dual map of the inclusion $j^{(0)}: {\mathcal H}^{(0)}\to\G^{(0)}$, where $X=G^{(0)}$ and $Y=H^{(0)}$. We claim that the action of $H$ on $X$, defined in the paragraph following \defnref{multiplier action}, is essentially free. This means that, introducing the semidirect product $H_X=H\ltimes X$ and its isotropy subgroupoid $H_X'$, $H_X'\setminus H_X^{(0)}$ has measure 0. Suppose that this is not true. Then there exists a Borel subset $B\subset H_X'\setminus H_X^{(0)}$ of positive measure. The inverse map 
\[\pi^{-1}: \B orel(Y)\to \B orel(X)\] 
induces a Boolean algebra homomorphism 
\[\pi^*: B(\mu_Y)\to B(\mu)\] 
where $\mu_Y=\pi_*\mu$, which is isometric and which extends $j^{(0)}$. Since $(X,\mu)$ is a Stonian measure space, $\G^{(0)}$ is identified to the measure algebra $B(\mu)$ (\cite[Corollary 1 of Chapter 41]{gh:boole}.  Since ${\mathcal H}^{(0)}$ is dense in $\G^{(0)}$ and the metric space $B(\mu_Y)$ is complete, $\pi^*$ is an isomorphism of $B(\mu_Y)$ onto  $B(\mu)$.  Using \lemref{idxpi}, one deduces that ${\rm id}\times\pi:H_X=H\ltimes X\to H\ltimes Y=H$ implements an isomorphism of the measure algebra of $(H,d^*\mu_Y)$ onto the measure algebra of $(H_X, d^*\mu)$. Therefore, there exists a Borel subset $B_1\subset H$ such that  $B\sim ({\rm id}\times\pi)^{-1}(B_1)$. Then, there exists a compact open bisection $S$ of $H$ such that $S_1=S\cap B_1$ has positive measure. For a.e. $(\eta,x)\in S_1*X$, we have $\eta x= x$. Equivalently, 
\[d^*\mu_Y(S_1\Delta d(S_1))=d^*\mu(\pi^{-1}(S_1)\Delta \pi^{-1}(d(S_1)))=0.\] 
Then, for $\mu$-a.e. $x\in\pi^{-1}(d(S_1))$ we have both $(S_1\pi(x), x)$ in $H^{(0)}$ and in its complement. This is a contradiction. We let $U$ be the set of points in $X$ with trivial isotropy. For $f\in C_c(H)$ and $(\eta,x)\in H*U$, we have $\tilde\j(f)(\eta x)=f(\eta)$. Therefore,
 \begin{align*}
 \int |\tilde\j(f)|^2 d(\mu\circ\tilde\lambda)=&\int_U \sum_{\gamma\in G_{x}}| \tilde\j(f)(\gamma)|^2 d\mu(x)\cr
 =&\int_U \sum_{\eta\in H_{\pi(x)}}| \tilde\j(f)(\eta x)|^2 d\mu(x)\cr
  =&\int \sum_{\eta\in H_{\pi(x)}} |f(\eta)|^2 d\mu(x)\cr
 =&\int \sum_{\eta\in H_y} |f(\eta)|^2 d\mu_Y(y) =\int |f|^2 d(\mu_Y\circ\tilde\lambda).\cr
\end{align*}
Consequently, $\tilde\j$ extends to an isometry $J: L^2(H, d^*\mu_Y)\to L^2(G, d^*\mu)$. Let us show that its image is dense. Since $\mathcal H$ is dense in ${\rm Bis}(G)$, for every $ S\in {\rm Bis}(G)$ and every $\epsilon>0$, there exists $T\in {\mathcal H}$ such that 
\[\rho_d( S,j(T))=d^*\mu(S\Delta j(T))\le \epsilon^2.\]
This says exactly that $\|{\bf 1}_{S}-{\bf 1}_{j(T)}\|_{L^2(G, d^*\mu)}\le\epsilon$. We deduce first that $J(C_c(Y))$ is dense in $L^2(X,\mu)$ and then that $J(C_c(H))$ is dense in $L^2(G, d^*\mu)$. For the last statement, we note that the Hilbert space $L^2(G, r^*\mu+d^*\mu)$, which is isometric to $L^2(G,d^*\mu)$, is also separable. Then, we consider a countable Boolean inverse subsemigroup $\mathcal H$ of ${\rm Bis}(G)$ which is dense with respect to the metric associated with the measure $r^*\mu+d^*\mu$.
 
\end{proof}

In the context of measured groupoids, the notion of continuous bisection may be replaced by the following definition (cf. \cite[Definition I.3.24]{ren:approach}).

\begin{defn}\label{non-singular}  Let $G$ be a Borel groupoid. A {\it Borel bisection} $S$ is a Borel subset of $G$ such that the restriction to it of the range and domain maps are Borel isomorphisms onto Borel subsets of the unit space. It is called {\it non-singular} with respect to a measure $\mu$ on $G^{(0)}$ if the measures $(r_{|S})^*\mu$ and $(d_{|S})^*\mu$ are equivalent.  
\end{defn}

Let us recall (\cite[Definition 1.3]{ren:twisted extensions}) that a representation of a locally compact twisted groupoid $(G,\Sigma)$ with Haar system $\lambda$  is a pair $(\mu,{\mathcal H})$ where $\mu$ is a quasi-invariant measure on $G^{(0)}$ and ${\mathcal H}$ is a measurable Hilbert bundle over a conull measurable subset $O$ of $G^{(0)}$ endowed with a measurable unitary representation $L$ of the reduction $\Sigma_{|O}$ and such that $L(\sigma\theta)\xi=L(\sigma)(\theta\xi)$ for all $(\theta, (\sigma,\xi))$ in $\t\times(\Sigma_{|O}*{\mathcal H})$.  Such a representation yields a representation of the $*$-algebra $C_c(G,\Sigma)$, still denoted by $L$, on the Hilbert space $L^2(G^{(0)},\mu;{\mathcal H})$ according to the formula
\[(\xi\,|\,L(f)\eta)=\int_G f(\sigma)(\xi\circ r(\sigma)\,|\, \,L(\sigma)(\eta\circ s(\sigma)))\,D^{-1/2}(\sigma')\,d\lambda^x(\sigma')d\mu(x).\]
where $D$ is the modular function of $\mu$,  $f$ belongs to $C_c(G,\Sigma)$ and $\xi,\eta$ belong to $L^2(G^{(0)},\mu;{\mathcal H})$. The representation $L$ extends through the same formula to the Banach $*$-algebra $L^I(G,\Sigma,\mu)$ whose elements are (classes of) measurable functions $f:\Sigma\to {\bf C}$ such that $f(\theta\sigma)=f(\sigma)\overline\theta$ for all $(\theta,\sigma)\in\t\times\Sigma$ and
\[\|f\|_I=\max\big(\sup_{x\in G^{(0)}}\int_G |f(\sigma)|d\lambda^x(\sigma'),\sup_{x\in G^{(0)}}\int_G |f(\sigma^{-1})|d\lambda^x(\sigma')\big)\]
is finite (where $\sup$ denotes here the essential supremum of the function). When $G$ is \'etale, $\lambda^x$ is the counting measure on $G^x=r^{-1}(x)$. We denote by $\nu_0=D^{-1/2}(r^*\mu)$ the symmetric measure on $G$.

\begin{lem}\label{approximation} Let $(\mu,L)$ be a representation of a twisted \'etale groupoid $(G,\Sigma)$. We assume that $c:G\to\Sigma$ is a measurable section of $p:\Sigma\to G$. Given a non-singular Borel bisection $S$ of $G$, we define $u_S:\Sigma\to{\bf C}$ by $u_S(\sigma)=\overline\theta$ if $\sigma$ can be written $\theta c(\gamma)$ with $\theta\in\t$ and $\gamma\in S$ and $u_S(\sigma)=0$ otherwise (in other words, $u_S=\Delta_{c(S)}$ with the notation of Section 1.3). Then,
\begin{enumerate}
 \item  $u_S$ belongs to $L^I(G,\Sigma,\mu)$,
 \item if $S_n,S$ are non-singular Borel bisections of $G$ such that $\nu_0(S_n\Delta S)$ tends to 0, then $L(u_{S_n})$ converges weakly to $L(u_S)$.
\end{enumerate}
\end{lem}

\begin{proof} We have $|u_S|={\bf 1}_S$ and $\|u_S\|_I\le 1$. We observe that if $S$ and $T$ are bisections of $G$, $|u_S-u_T|={\bf 1}_{S\Delta T}$ and that if $\xi$ and $\eta$ are elements of $L^2(G^{(0)},\mu;{\mathcal H})$,
\[|(\xi\circ r(\sigma)\,|\, \,L(\sigma)(\eta\circ s(\sigma)))|\le\sup_x\|\xi(x)\|\sup_y\|\eta(y)\| \]
Since $\|L(u_S)-L(u_{S_n})\|$ is bounded by $\|u_S-u_{S_n}\|_I\le 2$, it suffice to prove the convergence of $(\xi\,|\,(L(u_S)-L(u_{S_n}))\eta)$ to zero on the dense set of vectors $\xi$ in $L^2(G^{(0)},\mu;{\mathcal H})$ such that $\sup_x\|\xi(x)\| <\infty$. Then, $|(\xi |(L(u_S)-L(u_{S_n}))\eta)|$ is dominated by
\begin{align*}
&\int_G |u_S(\sigma)-u_{S_n}(\sigma)||(\xi\circ r(\sigma)\,|\, \,L(\sigma)(\eta\circ s(\sigma)))| d\nu_0(\sigma')\cr
&\le \sup_x\|\xi(x)\|\sup_y\|\eta(y)\| \int_G {\bf 1}_{S\Delta {S_n}}d\nu_0\cr
&\le \sup_x\|\xi(x)\|\sup_y\|\eta(y)\|\, \nu_0(S\Delta {S_n})\cr
\end{align*}
and tends to zero.
 \end{proof}

In order to apply our previous results, we note that a twist on a Stonian groupoid is a trivial principal $\t$-bundle.

\begin{prop}(cf. \cite[Proposition 4.6.]{dfp:cartan})\label{trivial twist}. Let $(G,\Sigma)$ be a twisted Stonian groupoid. Then the quotient map $p:\Sigma\to G$ admits a continuous section.
 \end{prop}
 
 \begin{proof} Let $(S_i)_{i\in I}$ be a maximal pairwise disjoint family of compact open bisections of $G$.  Its union is an open set which is dense by maximality. For every $i\in I$, we pick a continuous bisection $T_i$ of $\Sigma$ such that $p(T_i)=S_i$. For every $i$ and every $\gamma\in S_i$, there is a unique $\sigma\in T_i$ such that $p(\sigma)=\gamma$. We write $\sigma=c(\gamma)$. This defines a section $c: \cup_iS_i\to\Sigma$. To show that $c$ is continuous, it suffices to show that for all $i$ its restriction $c_{|S_i}$ is continuous. For every continuous bisection $U$ of $\Sigma$ and every open subset $W$ of $\t$,  
$c_{|S_i}^{-1}(UW)=p(UW\cap T_i)$ is open because $p$ is an open map. Therefore $c_{|S_i}$ is continuous is continuous and so is $c$. Let $S$ be a compact open bisection of $G$. The restriction of $c$ to $S\cap U$ is a continuous function into the compact space $p^{-1}(S)$. Therefore, according to \cite[0.1.1]{gut:bundles}, it admits a continuous extension $c_S: S\to p^{-1}(S)$ which is unique. By continuity $p\circ c={\rm id}_S$. If $S$ and $S'$ are two compact open bisections, $c_S$ and $c_{S'}$ agree on the dense open set $S\cap S'\cap U$ which is dense in $S\cap S'$, therefore $c_S$ and $c_{S'}$ agree on $S\cap S'$. We define $c:G\to\Sigma$ by $c(\gamma)=c_S(\gamma)$ if $\gamma$ belongs to the compact open section $S$. Then $c$ is a continuous section of $p$.
\end{proof}

Recall from the paragraph following \defnref{vN} that the Hilbert space $L^2(G, \Sigma, d^*\mu)$ of a twisted \'etale measure groupoid $(G,\Sigma,\mu)$ is defined as the completion of $C_c(G,\Sigma)$ with respect to the scalar product
\[(f\,|\, g)=\int\int(\overline fg)(\gamma)d\lambda_x(\gamma) d\mu(x).\]
where $\lambda_x$ is the counting measure on $G_x=d^{-1}(x)$ and $\overline f g$ is viewed as a function on $G$. We also recall from the paragraph following \defnref{twisted category} that a multiplier action of a twisted Boolean groupoid $(H,\Lambda)$ on another twisted Boolean groupoid $(G,\Sigma)$ defines a $*$-homomorphism $\tilde\j: C_c(H,\Lambda)\to C_c(G,\Sigma)$ which extends the twisted Boolean inverse semigroup homomorphism $j: ({\rm Bis}(H),{\rm Bis}(\Lambda))\to ({\rm Bis}(G),{\rm Bis}(\Sigma))$. The map $S\mapsto\Delta_S$ giving the embedding of ${\rm Bis}(\Sigma)$ into $C_c(G,\Sigma)$ has been defined in Section 1.3.

\begin{thm}\label{concrete} Let $(G,\Sigma, \mu)$ be a twisted Stonian measure groupoid such that the Hilbert space $L^2(G, \Sigma, s^*\mu)$ is separable. Then there exists a second countable twisted Boolean groupoid $(H,\Lambda)$ acting by multipliers on $(G,\Sigma)$ with anchor map $\pi: G^{(0)}\to H^{(0)}$ such that
\begin{enumerate}
 \item the $*$-homomorphism $\tilde\j: C_c(H,\Lambda)\to C_c(G,\Sigma)$ extends to a surjective isometry $J: L^2(H,\Lambda, d^*(\pi_*\mu))\to L^2(G,\Sigma, d^*\mu)$;
 \item  $J$ implements an isomorphism of $W^*(H,\Lambda,\pi_*\mu)$ onto $W^*(G,\Sigma,\mu)$ carrying $L^\infty(H^{(0)},\pi_*\mu)$ onto $L^\infty(G^{(0)},\mu)$.
\end{enumerate}
\end{thm}

\begin{proof} The existence of a continuous section $c$ for the quotient map $\Sigma\to G$, given in \propref{trivial twist} reduces part of the proof to the untwisted case. The Hilbert space $L^2(G,d^*\mu)$, which is isometric to $L^2(G, \Sigma, d^*\mu)$ is separable. According to \propref{main}, there exists a countable Boolean inverse subsemigroup $\mathcal H$ of $\G={\rm Bis}(G)$ which is dense with respect to the metric associated with the measure $d^*\mu+r^*\mu$.  We define ${\mathcal S}_{\mathcal H}$ as the inverse image of $\mathcal H$ under the map ${\mathcal S}\to\G$, where ${\mathcal S}={\rm Bis}(\Sigma)$. We let $(H,\Lambda)$ be the twisted groupoid of germs of $({\mathcal H},{\mathcal S}_{\mathcal H})$. The inclusion $j:({\mathcal H},{\mathcal S}_{\mathcal H})\to(\G,{\mathcal S})$ defines a multiplier action of $(H,\Lambda)$ on $(G,\Sigma)$, hence a $*$-homomorphism 
\[\tilde\j: C_c(H,\Lambda)\to C_c(G,\Sigma).\]The proof that for $f\in C_c(H,\Lambda)$, $\int |\tilde\j(f)|^2 d(d^*\mu)=\int |f|^2 d(d^*(\pi_*\mu))$ is the same as in \propref{main}. Therefore, $\tilde\j$ extends to an isometry 
\[J: L^2(H,\Lambda, s^*\pi_*\mu)\to L^2(G, \Sigma, s^*\mu).\]
Let us show that its image is dense.  Since $\mathcal H$ is dense in ${\rm Bis}(G)$, for every $S\in {\rm Bis}(G)$ and every $\epsilon>0$, there exists $T\in {\mathcal H}$ such that 
$s^*\mu(S\Delta j(T))\le \epsilon^2$.
Then we have
\[\|\Delta_{c(S)}-\Delta_{c(j(T))}\|_{L^2(G,\Sigma, s^*\mu)}=\|{\bf 1}_{ S}-{\bf 1}_{ j(T)}\|_{L^2(G, s^*\mu)}\le\epsilon.\] 
We deduce first that $J(C_c(H^{(0)}))$ is dense in $L^2(G^{(0)},\mu)$. Since $c(\G)$ and $L^2(G^{(0)},\mu)$ together generate $L^2(G,\Sigma, d^*\mu)$,  $J(C_c(H,\Lambda))$ is dense.
\vskip2mm
 By definition, the von Neumann algebra $W^*(H,\Lambda,\pi_*\mu)$ is the weak closure of the image of $C_c(H,\Sigma)$ by the left regular representation ${\rm Reg}_H(f)\xi=f*\xi$ on the Hilbert space $L^2(H,\Lambda,s^*\pi_*\mu)$. Similarly, $W^*(G,\Sigma,\mu)$ is the weak closure of the image of $C_c(G,\Sigma)$ by the left regular representation ${\rm Reg}_G$ on $L^2(G,\Sigma,s^*\mu)$. The equality \[J {\rm Reg}_H(f) J^{-1}={\rm Reg}_G(\tilde\j(f))\] for $f\in C_c(H,\Lambda)$ shows that $JW^*(H,\Lambda,\pi_*\mu)J^{-1}$ is contained in $W^*(G,\Sigma,\mu)$. Since  $j({\rm Bis}(H))$ is dense in ${\rm Bis}(G)$ for the metric associated with the symmetric measure $\nu_0$, \lemref{approximation} applied to the left regular representation (which is an integrated representation, see for example \cite[end of Section 2]{ren:twisted extensions} ) shows that $J {\rm Reg}_H(C_c(H,\Lambda))J^{-1}$ is weakly dense in $W^*(G,\Sigma,\mu)$. Therefore, we have the equality of the von Neumann algebras $J W^*(H,\Lambda,\pi_*\mu)J^{-1}$ and $W^*(G,\Sigma,\mu)$. 
\end{proof}

We can resume now our study of Cartan subalgebras in W*-algebras.

\begin{defn}
 We say that a measure Boolean groupoid $(G,\mu)$ is {\it principal} if the set of points in $G^{(0)}$ with non-trivial isotropy has measure zero.
\end{defn}
Note that this condition can be expressed as $d^*\mu(G'\setminus G^{(0)})=0$ where 
\[G'=\{\gamma\in G: r(\gamma)=d(\gamma)\}\] 
is the isotropy bundle of $G$.

\begin{lem}
 Let $(G,\mu)$ be a measure Boolean groupoid and let $A$ be a Borel subset of $G$ such that such that $d^*\mu(A\cap S)=0$ for every compact open bisection $S$. Then $d^*\mu(A)=0$. 
 \end{lem}

\begin{proof} We first observe that $d^*\mu(A\cap S)=0$ means that for $\mu$-a.e. $x$, $\lambda_x(A\cap S)=0$, where $\lambda_x$ is the counting measure on $G_x$. This implies that $\int {\bf 1}_A{\bf 1}_S(b\circ d) d(d^*\mu)=0$ for every continuous function $b$ on $G^{(0)}$. Since the elements of $C_c(G)$ are finite sums of functions of the form ${\bf 1}_S(b\circ d)$, where $S$ belongs to ${\rm Bis}(G)$ and $b$ belongs to $C(G^{(0)})$ and since $C_c(G)$ is dense in $L^2(G,d^*\mu)$, the condition $d^*\mu(A\cap S)=0$ for all $S$ in ${\rm Bis}(G)$ implies $\int {\bf 1}_A\,\xi \,d(d^*\mu)=0$ for every $\xi$ in $L^2(G,d^*\mu)$, hence $d^*\mu(A)=0$.
\end{proof}

\begin{prop}\label{principal} Let $(G,\Sigma,\mu)$ be a measure twisted Boolean groupoid. Then the subalgebra $L^\infty(G^{(0)},\mu)$ is maximal abelian in $W^*(G,\Sigma,\mu)$ if and only if $(G,\mu)$ is principal.
\end{prop}

\begin{proof} In fact, this result holds for arbitrary measure twisted groupoid (cf \cite[Theorem A.2.4]{adr:amenable}. The implication ``principal'' implies ``masa'' is part of \cite[Theorem 5.1]{hah:regular}. Let us prove the converse. Suppose that $(G,\mu)$ is not principal. Then, according to the lemma there exists a non-singular Borel bisection $S$ contained in $G'\setminus G^{(0)}$ of non-zero measure. There exists a non-singular Borel bisection $T$ of $\Sigma$ such that $p(T)=S$. We define $f(\sigma)=\overline\theta$ if $\sigma=\theta(Ts(\sigma))$ for some $\theta\in\t$ and 0 otherwise. By construction, $f$ is an element of $W^*(G,\Sigma,\mu)$. Since its support is contained in $G'$, it commutes with $L^\infty(G^{(0)},\mu)$. Therefore, $L^\infty(G^{(0)},\mu)$ is not maximal abelian.
\end{proof}

We obtain then the Feldman-Moore theorem as a corollary.

\begin{cor}\cite[Theorem 1]{fm:relations II}\label{fm} Let $A$ be a W*-algebra with a separable predual having a Cartan subalgebra $B$. Then there exists a second countable twisted Boolean groupoid $(G,\Sigma)$, a quasi-invariant measure $\mu$ and an isomorphism of $A$ onto $W^*(G,\Sigma,\mu)$ carrying $B$ onto $L^\infty(G^{(0)},\mu)$. Moreover the measure groupoid $(G,\mu)$ is principal.
 \end{cor}

\begin{proof} Since $B$ has also a separable predual, it is countably decomposable. Therefore, it has a faithful normal state $\m$. The abstract realization gives the isomorphism of the Cartan pairs $(A,B)$ and $W^*(G(B),\Sigma(B),\underline\mu), L^\infty(\hat B,\underline\mu))$ where $\underline\mu$ is the measure extending $\m$. Since the GNS representation defined by ${\m}\circ p$ is faithful, the Hilbert space $L^2(G(B),\Sigma(B),s^*\underline\mu)$ is separable. We apply \thmref{concrete} to construct a second countable twisted Boolean groupoid $(G,\Sigma)$ acting on $(G(B),\Sigma(B))$ by multipliers and such that the ample semigroup of $G$ is dense  in the Weyl pseudogroup $\G(B)$ for the metric associated with $r^*\underline\mu+d^*\underline\mu$. Then \thmref{concrete} says that the pair $W^*(G(B),\Sigma(B),\underline\mu), L^\infty(\hat B,\underline\mu))$ is isomorphic to the pair $W^*(G,\Sigma,\mu), L^\infty(G^{(0)},\mu))$ where $\mu$ is the image of $\underline\mu$ by the anchor map of the action. Since $B$ is maximal abelian, so is $L^\infty(G^{(0)},\mu)$. Therefore, $(G,\mu)$ is a principal measure groupoid. 
\end{proof}

\appendix

\section {Non-commutative Stone equivalence}

Classical Stone duality between Boolean algebras and Boolean spaces admits a generalization to an equivalence between the category of Boolean inverse semigroups and the category of Boolean groupoids. The arrows of the first category are the Boolean inverse semigroup morphisms. The arrows of the second category are the multiplier actions with proper anchor, defined below. The equivalence will be extended to the twisted categories. We refer the reader to \cite{law:inverse} for the theory of inverse semigroups and to \cite{ren:approach} or \cite{pat:gpd} for the theory of locally compact groupoids. We recall some basic definitions and properties. 

\begin{defn}
 An {\it inverse semigroup} $\G$ is a semigroup such that for each $s\in\G$ there
exists a unique $t\in\G$ such that $sts=s$ and $tst=t$. One calls $t$ the
inverse of $s$ and writes
$t=s^*$. One denotes $\G^{(0)}=\{e\in\G:e^2=e\}$ the set of idempotents.
\end{defn}

Note that  idempotents $e$ are necessarily self-adjoint: $e^*=e$. We write $d(s)=s^*s$ and $r(s)=ss^*$; they belong to $\G^{(0)}$. The relation $s\le t$ if $s=te$ for some $e\in\G^{(0)}$ is an order relation on $\G$. The poset $(\G^{(0)},\le)$ is a meet semilattice with $e\wedge f=ef$. Assuming that $\G$ has a zero element, we say that a family $(s_i)$ of elements of $\G$ is orthogonal if for all $i\not=j$, $s^*_is_j=0$.
\begin{defn}\label{boolean isg}  An inverse semigroup $\G$ is called {\it Boolean} if
\begin{enumerate}
 \item $\G^{(0)}$ is a Boolean algebra;
 \item every orthogonal pair $(s,t)$ in $\G$ has a join $s\vee t$ (also written $\sup(s,t)$) in $\G$.
\end{enumerate}
\end{defn}

When $\G$ is a Boolean inverse semigroup, we implicitly assume that $\G^{(0)}$ is the Boolean algebra $\B(X)$, where $X$ is the dual space of $\G^{(0)}$. Given $s\in\G$, we denote by ${\rm dom}(s)$ [resp. ${\rm ran}(s)$] rather than $d(s)$ [resp. $r(s)$] the compact open subset of $X$ which is the support of the idempotent $d(s)$ [resp. $r(s)$].
 
\begin{defn} The {\it category of Boolean inverse semigroups} is defined as the category which has the Boolean inverse semigroups as objects and the Boolean inverse semigroups morphisms as arrows.
\end{defn}

\begin{defn}
 An {\it \'etale groupoid} $G$ is a locally compact groupoid such that its range and domain maps $r,d:G\to G^{(0)}$ are local homeomorphisms.
\end{defn}

\begin{defn}\label{boolean gpd}  A {\it Boolean groupoid} is an \'etale groupoid $G$ whose unit space $G^{(0)}$ is locally compact Hausdorff and totally disconnected.  
\end{defn}

The definition of the arrows of the category of Boolean groupoids requires some familiar notions of the theory of locally compact groupoids which we recall from \cite[Definition 3]{ren:rieffel} and \cite[Definition 2.1]{hol:I} and which can also be found in the recent preprint \cite{tay:functoriality}. A left {\it action of a groupoid} $H$ on a set $X$ is given by a map $\rho: X\to H^{(0)}$, called {\it anchor} or {\it moment map}, and an {\it action map} $H*X\to X$, where $H*X$ is the set of composable pairs $(\eta,x)\in H\times X$ with $d(\eta)=\rho(x)$. The image of $(\eta,x)$ is denoted by $\eta x$. By definition $\rho(x)x=x$ for all $x\in X$, $\rho(\eta x)=r(\eta)$ for all $(\eta,x)\in H*X$ and $\eta(\eta'x)=(\eta\eta')x$ for all $(\eta,\eta',x)$ in $H*H*X$. One says then that $X$ is a left $H$-space. One defines similarly a right $G$-space. If $H$ and $G$ are two groupoids and $X$ is both a left $H$-space and a right $G$-space such that $\eta(x\gamma)=(\eta x)\gamma$ for all $(\eta,x\gamma)\in H*X*G$, one says that $X$ is an {\it $(H,G)$-space}. In the locally compact setting, we assume that $X$ and $G$ are locally compact, the moment map is continuous (openness which is assumed in \cite{ren:rieffel} is not needed here) and the action map is continuous. One says that the right $G$-space $X$ is {\it free} if the map from $X*G$ to $X\times X$ sending $(x,\gamma)$ to $x\gamma$ is one-to-one, {\it proper} if this map is  proper and {\it principal} if it is free and proper. A {\it correspondence} is an $(H,G)$-space $X$ such that the $G$-action is principal. The following particular case appears in \cite{bs:morphisms}, \cite{mz:categories} and in the recent preprint \cite{tay:functoriality}.

\begin{defn}\label{multiplier action} A {\it multiplier action} of a locally compact groupoid $H$ on a locally compact groupoid $G$ is a correspondence $X$ from $H$ to $G$ where $X=G$ and $G$ acts by multiplication on the right.
\end{defn}

Multiplier actions are called actors in \cite{mz:categories,tay:functoriality}. Proposition 4.16 of \cite{mz:categories} gives the following decomposition of a multiplier action. If $X$ is a locally compact left $H$-space, then $H$ acts by multipliers on the semidirect product $G=H\ltimes X$. The action is given by $\eta(\eta',x)\defequal(\eta\eta',x)$. On the other hand, if $H$ and $G$ have the same unit space and $F:H\to G$ is a groupoid homomorphism whose restriction to the unit space is the identity map, then
$H$ acts on $G$ by multipliers according to the action map $\eta\gamma\defequal F(\eta)\gamma$. The most general case is a composition of these two cases. Suppose indeed that $H$ acts on $G$ by multipliers. Then the map $\rho:G\to H^{(0)}$ factors through the range map of $G$ because of the equality $\eta\gamma=\eta r(\gamma)\gamma$. Defining $\rho^{(0)}$ as the restriction of $\rho$ to $G^{(0)}$, we have $\rho(\gamma)=\rho^{(0)}\circ r(\gamma)$. If we let $H$ act on $G^{(0)}$ by the moment map $\rho^{(0)}$ and the action map defined by $\eta\cdot x=r(\eta x)$, then the map $F$ of the semidirect product $H\ltimes G^{(0)}$ into $G$ sending $(\eta,x)$ to $\eta x$ is a groupoid homomorphism whose restriction to the unit space is the identity map. One can retrieve the multiplier action of $H$ on $G$ by the formula \[\eta\gamma=F(\eta,r(\gamma))\gamma.\]

\begin{defn}\label{proper anchor} We shall say that a multiplier action of a locally compact groupoid $H$ on a locally compact groupoid $G$ has {\it proper anchor} if the moment map $\rho^{(0)}: G^{(0)}\to H^{(0)}$ is a proper map.
\end{defn}

\vskip2mm
When $H$ and $G$ are \'etale with respective Haar systems $\beta$ and $\lambda$ of counting measures, a multiplier action of $H$ on $G$ gives a multiplier action of the $*$-algebra $C_c(H)$ on the $*$-algebra $C_c(G)$ according to the convolution product
\[f*_\beta g(\gamma)=\int f(\eta)g(\eta^{-1}\gamma)d\beta^{\rho(\gamma)}(\eta)\]
for $f\in C_c(H)$ and $g\in C_c(G)$. This is a particular case of the construction of correspondences given by Holkar (see \cite{hol:I} and \cite{ren:rieffel}). If moreover the moment map $\rho^{(0)}: G^{(0)}\to H^{(0)}$ is proper, one has a $*$-homomorphism of $C_c(H)$ into $C_c(G)$ given by
\[\tilde \j(f)(\gamma)=\sum_{\eta\in H:\eta d(\gamma)=\gamma} f(\eta)=f*_\beta{\bf 1}_{G^{(0)}}(\gamma).\] 
In fact, we still have $\tilde\j(f)=f*_\beta h$ for any $h\in C_c(G^{(0)})$ which takes the value 1 on $(\rho^{(0)})^{-1}(s({\rm supp}(f))$.

\vskip2mm
Let $G,H,K$ be locally compact groupoids. If $H$ acts on $G$ by multipliers with moment map $\rho$ and $K$ acts on $H$ with moment map $\sigma$, then $K$ acts on $G$ by multipliers with moment map $\sigma^{(0)}\circ \rho$ and action $\kappa\gamma\defequal (\kappa\rho(\gamma))\gamma$ for $(\kappa,\gamma)\in K\times G$ such that $\sigma^{(0)}\circ \rho(\gamma)=s(\kappa)$. We note that if $\rho^{(0)}$ and $\sigma^{(0)}$ are proper, so is $\sigma^{(0)}\circ\rho^{(0)}$. This composition is a particular case of the composition of groupoid correspondences given in \cite{ren:rieffel} and \cite{hol:II}. By choosing as objects the locally compact groupoids and as arrows the multiplier actions, one gets a category, rather than a weak category.

\begin{defn} The {\it category of Boolean groupoids} is the category whose objects are the Boolean groupoids and whose arrows are the multiplier actions with proper anchor.
\end{defn}

Given a Boolean groupoid $G$, we define
\[{\rm Bis}(G)\defequal\{\hbox{ compact open bisections of}\quad G\}.\]

\begin{prop} Endowed with the set-theoretical product, the set ${\rm Bis}(G)$  of compact open bisections of a Boolean groupoid $G$ is a Boolean inverse semigroup, called the ample inverse semigroup of $G$. 
\end{prop}

\begin{proof} It is well known that the compact open bisections of a Boolean groupoid form an inverse semigroup $\G$ (see \cite[Definition I.2.10]{ren:approach} and \cite[Proposition 2.2.4]{pat:gpd}). The idempotents of $\G$ are the compact open subsets of $G^{(0)}$. If $E$ and $F$ are compact open subsets of $G^{(0)}$, then $EF=E\cap F$. One deduces that the order of the inverse semigroup $\G$ is the inclusion. As observed in \cite[page 142]{ren:approach}, if $(S,T)$ is an orthogonal pair of compact open bisections, which means here that $r(S)\cap r(T)=\emptyset$ and $d(S)\cap d(T)=\emptyset$, then $S\cup T$ is a compact open bisection. This is necessarily the join $S\vee T$ of $S$ and $T$. Moreover, the compact open subsets of $G^{(0)}$ form a Boolean algebra.
\end{proof}

The definition of the groupoid of germs of a Boolean inverse semigroup $\G$ has been given in \defnref{germ}:
\[{\rm Germ}(\G)\defequal \G*X/\sim=\{[s,x]: s \in\G, x\in{\rm dom}(s)\}.\]
where $X$ is the dual space of the Boolean algebra $\G^{(0)}$ and $\G$ acts on $X$ by dualizing its action on $\G^{(0)}$ by conjugation $e\mapsto ses^*$. Since, as any groupoid of germs, it is \'etale and its unit space is a Boolean space, ${\rm Germ}(\G)$ is a Boolean groupoid.

\begin{thm}\label{equivalence} The constructions of the groupoid of germs of a Boolean inverse semigroup and of the ample semigroup of a Boolean groupoid are functorial and give an equivalence of the category of Boolean inverse semigroups and the category of Boolean groupoids.
\end{thm}

\begin{proof} We recall that ${\rm Germ}(\G)$ denotes the groupoid of germs of a Boolean inverse semigroup $\G$ and ${\rm Bis}(G)$ the inverse semigroup of compact open bisections of a Boolean groupoid $G$. We first show that ${\rm Bis}$ and ${\rm Germ}$ are functors. Suppose that $H$ acts on $G$ with proper anchor  $\rho^{(0)}$. Let $S$ be a compact open bisection of $H$. Define
\[j(S)\defequal SG^{(0)}=\{\eta x: \eta\in S, x\in G^{(0)},\, \rho^{(0)}(x)=d(\eta)\}.\]
Since, necessarily $\eta=S\rho^{(0)}(x)$, $j(S)$ is a section for the domain map of $G$. The relation $(\eta x)^{-1}=\eta^{-1}(\eta.x)$ shows that $j(S)^{-1}=j(S^{-1})$. Therefore, $j(S)$ is also a section for the range map of $G$. Since $j(S) $ is the image of the compact open set $S*(\rho^{(0)})^{-1}(s(S))$ by the action map $H*G\to G$ which is continuous and open (according to \cite[Lemma 2.4]{tay:functoriality}), it is compact and open. Thus $j$ is a map from ${\rm Bis}(H)$ to ${\rm Bis}(G)$. We also have $j(ST)=STG^{(0)}=Sj(T)=SG^{(0)}j(T)=j(S)j(T)$. Therefore, $j$ is an inverse semigroup homomorphism from ${\rm Bis}(H)$ to ${\rm Bis}(G)$. Its restriction to the idempotents, which is $(\rho^{(0)})^{-1}$, is a Boolean algebra homomorphism. Let $j$ be a homomorphism from a Boolean inverse semigroup ${\mathcal H}$ into a Boolean inverse semigroup $\G$. Let $H$ and $G$ the respective groupoids of germs. The Boolean algebra homomorphism $j^{(0)}$ gives a continuous and proper map $\rho^{(0)}: G^{(0)}\to H^{(0)}$. Then we define a multiplier action of $H$ on $G$ with proper anchor $\rho^{(0)}$ by
\[ [s,\rho^{(0)}\circ\varphi_t(x)][t,x]=[j(s)t,x]\]
for $s\in{\mathcal H}$, $t\in\G$ and $x\in{\rm dom}(t)$, where we have used the notation given before \defnref{germ}. The algebraic properties of a multiplier actions are easily verified. Since the restriction of the action map to the product of  the open bisections $S(s)\subset H$ and $S(t)\subset G$ is clearly continuous, so is the action map. We skip the easy but tedious task to check that these constructions respect the composition of arrows. 
\vskip2mm
Let us show that the Boolean groupoid $G$ is isomorphic to the groupoid of germs ${\rm Germ}(\G)$ of its ample semigroup $\G={\rm Bis}(G)$. The map $f: {\rm Germ}(\G)\to G$ sending $[S,x]$ to $Sx$, where $S\in\G$ and $x\in d(S)$, is well defined: if $(S,x)\sim (T,x)$, then $Sx=Tx$. The map is injective: if $Sx=Ty$, then $x=d(Sx)=d(Ty)=y$. Then $d(S\cap T)$ is an open neighborhood of $x$. If $E$ is a compact open neighborhood of $x$ contained in $d(S\cap T)$, we have $SE=TE$, hence $[S,x]=[T,y]$. The map is surjective: given $\gamma\in G$, there exists a compact open bisection $S$ containing $\gamma$. Then $\gamma=Sd(\gamma)=f([S,d(\gamma)])$. The definition of the product in ${\rm Germ}(\G)$, namely
\[ [S, \varphi_T(x)][T,x]=[ST,x]\]
makes clear that $f$ is a groupoid homomorphism. If $S$ is a compact open bisection of $G$, $f^{-1}(S)=\{[S,x], x\in d(S)\}$ is a basic open set in ${\rm Germ}(\G)$. Therefore, $f$ is a homeomorphism.  One deduces that $f$ is an isomorphism of topological groupoids. 
\vskip2mm
Let us show that the ample semigroup of the groupoid of germs $G={\rm Germ}(\G)$ of a Boolean inverse semigroup $\G$ is isomorphic to $\G$. We recall that for all $s\in\G$, \[S(s)=\{[s,x]: x(s^*s)=1\}\] is a bisection of $G$ which is open by definition of the topology of germs. Since it is homeomorphic to the compact open subset ${\rm dom}(s)$ of the dual space $X$ of $\G^{(0)}$, it is also compact. The property $S(s)S(t)=S(st)$ is easily checked: an element of the first set is of the form $[s,y][t,x]=[st,x]$ where $x\in{\rm dom}(t)$ and $y=\varphi_t(x)\in{\rm dom}(s)$; an element of the second set is of the form $[st,x]$ where $x\in{\rm dom}(st)$. Since the condition $x((st)^*st)=1$ can be written $x\in{\rm dom}(t)$ and $\varphi_t(x)(s^*s)=1$, the two sets agree. If $e\in\G^{(0)}$, $S(e)$ is the compact open subset ${\rm dom}(e)$ of $X$ corresponding to $e$. Let us show that the map $s\mapsto S(s)$ is injective. Assume that $S(s)=S(t)$. First, this implies that ${\rm dom}(s)={\rm dom}(t)$. For all $x\in {\rm dom}(s)$, we have $[s,x]=[t,x]$, hence there exists $e\in\G^{(0)}$ such that $x(e)=1$ and $se=te$. Being compact, ${\rm dom}(s)$ has a finite cover $(e_i)$ such that $se_i=te_i$ for all $i$. Since $\G^{(0)}$ is a Boolean algebra, we can assume that this cover is a partition of ${\rm dom}(s)$. If $i\not= j$, the pair $(se_i,se_j)$ is orthogonal. The join $se_i\vee se_j$ exists and we have $s(e_i\vee e_j)=se_i\vee se_j$. Here is the standard proof (see \cite{law:inverse}) of this equality:
\begin{align*}
e_i\le e_i\vee e_j &\Rightarrow se_i\le s(e_i\vee e_j)\cr
&\Rightarrow se_i\vee se_j\le s(e_i\vee e_j)
\end{align*}
With $t=se_i\vee se_j$, we also have
\begin{align*}
se_i, se_j\le t&\Rightarrow e_i,e_j\le s^*t\cr
&\Rightarrow e_i\vee e_j\le s^*t\cr
&\Rightarrow s(e_i\vee e_j)\le ss^*t\le t\cr
\end{align*}
We deduce that
\[s=s(e_1\vee\ldots\vee e_n)=se_1\vee\ldots\vee se_n=te_1\vee\ldots\vee te_n=t(e_1\vee\ldots\vee e_n)=t.\]
Let us show that the map $s\mapsto S(s)$ is surjective. Let $S$ be a compact open bisection of $G$. Let $\gamma\in S$. Since $\{S(s), s\in\G\}$ is a base of open sets of $G$, there exists $s\in\G$ such that $S(s)$ contains $\gamma$ and is contained in $S$. By compactness of $S$, there exists a finite family $(s_i)$ in $\G$ such that $S$ is the union of the family$(S(s_i))$. Then $S(s_i)=SE_i$, where $E_i={\rm dom}(s_i)$. Since $\G^{(0)}$ is a Boolean algebra, we can assume that the family $(s_i^*s_i)$ is orthogonal, or equivalently, that the family $(E_i)$ is a partition of $d(S)$. Then $S$ is the disjoint union of the family$(S(s_i))$ and the family $(s_i)$ is orthogonal. Since $\sup s_i$ exists, $S=S(\sup s_i)$.

\end{proof}

\begin{rem} A very similar result is given in \cite[Propositions 5.3 and 5.4]{exe:isg} (see also \cite{be:fell, bem:isa}) for arbitrary inverse semigroups $S$ and \'etale groupoids $G$. There, ${\rm Bis}(G)$ is the inverse semigroup of the open bisections of $G$ while ${\rm Gr}(S)$ is the groupoid of germs of the action of $S$ on the character space of the semi-lattice $S^{(0)}$. 
\end{rem}

\thmref{equivalence} shows that Boolean inverse semigroups and Boolean groupoids have essentially the same theory. Any property of the former corresponds to a property of the latter and one can establish a dictionary of these properties.  Here are some examples:

\begin{prop}\label{properties}  Let $G$ be a Boolean groupoid, $\G={\rm Bis}(G)$ and $X=G^{(0)}$.
\begin{enumerate}
 \item $G$ is {\it effective} (i.e. the interior of its isotropy is reduced to its units) if and only $\G$ is {\it fundamental} (i.e. is a pseudogroup of transformations of $X$);
 \item $G$ is a group bundle (i.e. $r=d$) if and only if $\G$ is {\it Clifford} (i.e. $r=d$);
 \item $G$ is second countable if and only if $\G$ is countable;
 \item $G$ is Hausdorff if and only $\G$ is closed under finite intersections.
 \end{enumerate}
\end{prop}

\begin{proof} For (i), see \cite[Corollary 3.3]{ren:cartan}. The assertions (ii) and (iii) are easy. The assertion (iii) is part of \cite[Proposition 3.7]{ste:isga}.
 
\end{proof}

\subsection*{Twisted non-commutative Stone equivalence} We show in this section that \thmref{equivalence} extends to the twisted categories. Since we have to consider topological groupoids which are not \'etale as the trivial group bundle $X\times\t$ where $X$ is a locally compact Hausdorff space and $\t$ is the group of complex numbers of modulus one, we first extend the notion of open bisection.

\begin{defn}\label{continuous bisection} A bisection $S$ of a topological groupoid $G$ is said to be {\it continuous} if the restrictions $r_{|S}$ and $d_{|S}$ are homeomorphisms onto open subsets of $G^{(0)}$. 
\end{defn}

\begin{prop} Let $G$ be an \'etale groupoid. Then a bisection $S$ of $G$ is continuous if and only if it is open.
 \end{prop}

\begin{proof} It suffices to consider the range map.
 If $S$ is open, $r(S)$ is open and $r_{|S}:S\to r(S)$ is injective, continuous and open, hence a homeomorphism.
 Let us assume conversely that $V=r(S)$ is open and that $r_{|S}:S\to V$ is a homeomorphism. Then the restriction of $r$ to $r^{-1}(V)$ is an \'etale map onto $V$ and the inverse $s$ of $r_{|S}$ is a continuous section. According to \cite[Corollary 3, I.30]{bbki:ta}, its image $S$ is open in $r^{-1}(V)$, hence in $G$.
\end{proof}
 
 \begin{prop} Let $G$ be a locally compact groupoid. Then
 
\begin{enumerate}
 \item The continuous bisections form a inverse subsemigroup of the inverse semigroup of bisections of $G$.
 \item The compact continuous bisections form a inverse subsemigroup of the inverse semigroup of bisections of $G$ which is Boolean if $G^{(0)}$ is Boolean.
\end{enumerate}
\end{prop}

\begin{proof}
 For (i), let $S,T$ be continuous bisections. Then $r_{|ST}$ can be decomposed as
 $$ST\to Sr(T)\to r(ST)=r(Sr(T))$$
 The first map is isomorphic to $d(S)T\to d(S)\cap r(T)$, which is the restriction of $r$ to an open subset of $T$ and the second is the restriction of $r$ to an open subset of $S$. Moreover, if $S$ is a continuous bisection, so is obviously $S^{-1}$.
 
 For (ii), it suffices to say that $ST$ is the image of the compact $S\times T$ by the product map $G^{(2)}\to G$ which is continuous and that $S^{-1}$ is the image of $S$ by the inverse map which is also continuous.
\end{proof}

\begin{defn}\label{Bisbis} The {\it ample inverse semigroup} of a locally compact groupoid $G$ such that $G^{(0)}$ is Boolean is defined as the inverse semigroup of its compact continuous bisections. It is denoted by ${\rm Bis}(G)$.
 \end{defn}
 
 \begin{defn}\label{trivial Clifford} The {\it trivial Clifford inverse semigroup} over the Boolean space $X$, denoted by ${\mathcal T}(X)$, is the ample inverse semigroup ${\rm Bis}(X\times\t)$ of the trivial group bundle $X\times\t$.
\end{defn}

We define now the category of twisted Boolean groupoids and the category of twisted Boolean inverse semigroups.

\begin{defn}\label{twisted category} The objects of the {\it category of twisted Boolean groupoids} are the twisted Boolean groupoids $(G,\Sigma)$ as defined in \defnref{gpd twist}. An arrow from $(H,\Lambda)$ to $(G,\Sigma)$ is a $\t$-equivariant multiplier action with proper anchor of $\Lambda$ on $\Sigma$.
 \end{defn}
By definition, an action of $\Lambda$ on $\Sigma$ is $\t$-equivariant if $(\theta l)\sigma=l(\theta\sigma)$ for all $\theta\in\t$ and all $(l,\sigma)\in\Lambda*\Sigma$. Such an action induces a multiplier action of $H$ on $G$. Then $C_c(H,\Lambda)$ acts on $C_c(G,\Sigma)$ by multipliers (these convolution algebras have been defined in Section 1.3). The formula is as above:
\[f*_\beta g(\sigma)=\int f(l)g(l^{-1}\sigma)d\beta^{\rho(\sigma)}(l')\]
for $f\in C_c(H,\Lambda)$ and $g\in C_c(G,\Sigma)$ and where $l'$ is the image of $l$ in $H$. Since we assume that the moment map $\rho^{(0)}: G^{(0)}\to H^{(0)}$ is proper, we have a $*$-homomorphism of $C_c(H,\Lambda)$ into $C_c(G,\Sigma)$ given by
\[j(f)(\sigma)=\sum_{l\in \Lambda:l s(\sigma)=\sigma} f(l)=f*_\beta\Delta_{G^{(0)}}(\sigma)\] 
(where $\Delta_S$ is defined in Section 1.3 and before \defnref{refine}) such that $f*_\beta g=j(f)*_\lambda g$.

\begin{defn} The objects of the {\it category of twisted Boolean inverse semigroups} are the twisted Boolean inverse semigroups $(\G,{\mathcal S})$ as defined in \defnref{isg twist}. An arrow from $({\mathcal H},{\mathcal S}_{\mathcal H})$ to $(\G,{\mathcal S})$ is a pair of Boolean inverse semigroup morphisms $j:{\mathcal H}\to {\mathcal G}$ and $j_1:{\mathcal S}_{\mathcal H}\to {\mathcal S}$ making the following diagram commutative
\[\begin{CD}
{\mathcal T}_{{\mathcal H}^{(0)}}@>{i_{\mathcal H}}>{}>{\mathcal S}_{\mathcal H}@>{p_{\mathcal H}}>{}>{\mathcal H}\\
 @V{}V{j_1^{(0)}}V  @V{}V{j_1}V@V{}V{j}V\\
{\mathcal T}@>{ i}>{}>{\mathcal S}@>{ p}>{}>\G
 \end{CD}
\]
where $j_1^{(0)}:{\mathcal T}_{{\mathcal H}^{(0)}}\to{\mathcal T}$ is the natural extension of $j^{(0)}:{\mathcal H}^{(0)}\to \G^{(0)}$. Equivalently, it is a Boolean inverse semigroup morphism $j_1:{\mathcal S}_{\mathcal H}\to {\mathcal S}$ whose restriction to ${\mathcal T}_{{\mathcal H}^{(0)}}$ is the natural extension of $j^{(0)}:{\mathcal H}^{(0)}\to \G^{(0)}$.

\end{defn}
\begin{prop}\label{gpdtwist to isgtwist} Let  $(G,\Sigma)$ be a twisted Boolean groupoid. Then $({\rm Bis}(G), {\rm Bis}(\Sigma))$ is a twisted Boolean inverse semigroup.
\end{prop}

\begin{proof} Let $p:\Sigma\to G$ be the quotient map. If $S$ is a continuous bisection of $\Sigma$, then $p(S)$ is a continuous bisection of $G$. Since $G$ is \'etale, $p(S)$ is an open bisection of $G$. If $S$ is compact, so is $p(S)$. This defines a map from ${\rm Bis}(\Sigma))$ to ${\rm Bis}(G)$ which we denote also by $p$. The relation $p(ST)=p(S)p(T)$ and $p(S^{-1})=p(S)^{-1}$ are easily checked. Let us show that $p: {\rm Bis}(\Sigma))\to{\rm Bis}(G)$ is surjective. Let $S$ be a compact open bisection of $G$. Since  $p:\Sigma\to G$ admits continuous local sections (see \cite{ren:twisted extensions}, paragraph following Definition 1.1), there is an open cover $(U_i)$ of $S$ and continuous sections $\sigma_i$ defined on $U_i$. Since $S$ is a compact Boolean space, we can assume that the $U_i$'s are clopen and that the cover $(U_i)$ is a finite partition;  we define a continuous section $\sigma$ on $S$ by $\sigma_{|U_i}=\sigma_i$. Then $\sigma(S)$ is a compact continuous bisection of $\Sigma$ such that $p(\sigma(S))=S$. Let us show that the kernel of the map $p: {\rm Bis}(\Sigma)\to{\rm Bis}(G)$ is the trivial Clifford inverse semigroup ${\mathcal T}(G^{(0)})$. Let $S$ be a compact continuous bisection of $\Sigma$ such that $p(S)\subset G^{(0)}$. Then $S$ is a continuous (bi)section of $G^{(0)}\times\t$. In other words, $S$ belongs to ${\mathcal T}(G^{(0)})$. It is clear that if $S$ is a continuous (bi)section of $G^{(0)}\times\t$, then $p(S)\subset G^{(0)}$.
\end{proof}

The following proposition is essentially \cite[Theorem 3.22]{be:fell}, where it is expressed in terms of Fell line bundles rather than twists. The direct construction of the twisted Boolean groupoid is outlined after \cite[Definition 3.23]{be:fell}.
We recall that we define on ${\mathcal S}*X=\{(s,x): s\in {\mathcal S}, x\in{\rm dom}(s)\}$ the equivalence relations
 $(s,x)\sim (t,y)$ if and only if
  $$x=y\quad{\rm and}\quad\exists\, a,b\in {\mathcal T}(X): a(x)=b(x)=1\quad{\rm and}\quad sa=tb$$
 and
 $(s,x)\approx (t,y)$ if and only if
  $$x=y\quad{\rm and}\quad\exists\, a,b\in {\mathcal T}(X): |a(x)|=|b(x)|=1\quad{\rm and}\quad sa=tb.$$
  The equivalence class of $(s,x)$ for $\sim$ [resp. for $\approx$] is denoted by $[s,x]$ [resp. by $[[s,x]]$].

\begin{prop}\label{isgtwist to gpdtwist} Let $(\G,{\mathcal S})$ be a twisted Boolean inverse semigroup. Then,
\begin{enumerate}
 \item $\Sigma={\mathcal S}*X/\sim$ has a groupoid structure with
$$[s,\varphi_t(x)][t,x]=[st,x],\quad [s,x]^{-1}=[s^*,\varphi_s(x)].$$
 \item $G={\mathcal S}*X/\approx$ has a groupoid structure with
$$[[s,\varphi_t(x)]][[t,x]]=[[st,x]],\quad [[s,x]]^{-1}=[[s^*,\varphi_s(x)]].$$
\item the map $p:\Sigma\to G$ sending $[s,x]$ to $[[s,x]]$
is a surjective groupoid homomorphism such that $p^{(0)}$ is the identity of $X$ and ${\rm Ker}(p)$ is isomorphic to $X\times\t$.
\item $(G,\Sigma)$ is a twisted Boolean groupoid.
\item $(\G,{\mathcal S})$ is isomorphic to $({\rm Bis}(G), {\rm Bis}(\Sigma))$.
\end{enumerate}
 \end{prop}

\begin{proof}
The proofs of (i) and (ii) are similar to that of \cite[Section 4]{exe:isg}. The unit spaces of $G$ and $\Sigma$ are the dual space $X$ of the Boolean algebra $\G^{(0)}$. The domain and range maps are respectively $d([s,x])=x$, $r([s,x])=\varphi_s(x)$. Note also that $G$ is the groupoid of germs of $\G$. For (iii), the surjectivity of $p$ is clear, as well as the fact that it is a groupoid homomorphism. The class in $\Sigma$ of $(c,x)$, where $c\in{\mathcal T}(X)$ and $x\in{\rm dom}(c)$ depends only on $(x,c(x))$; moreover, every $(x,\theta)\in X\times\t$ can be realized as $(x,c(x))$ for some $c\in{\mathcal T}(X)$. This defines an injective homomorphism $j:X\times\t\to\Sigma$ and, since $(c,x)\approx (|c|,x)$, $p\circ j(x,\theta)$ is a unit. Let $(s,x)\in{\mathcal S}*X$ such that $[[s,x]]$ is a unit. There exist $a,b\in {\mathcal T}(X)$ such that $|a(x)|=|b(x)|=1$ and $sa=b$. Let $c=ba^*\in {\mathcal T}(X)$. Then $(s,x)\sim (c,x)$. For (iv), it results from the previous assertions that $(G,\Sigma)$ is a twist. In particular, $\Sigma$ is endowed with a free action of $\t$, namely $[s,x]\theta=[sc,x]$ where $c\in{\mathcal T}(X)$ satisfies $c(x)=\theta$ and $p:\Sigma\to G$ is the quotient map. We endow $G$ with the topology of germs. In order to define the topology of $\Sigma$, we have to take into account the topology of $\t$. Given $s\in{\mathcal S}$, we define as earlier
 \[S(s)=\{[s, x]: x\in{\rm dom}(s)\}.\] Let us show that the family $\{S(s)W\}$, where $s\in{\mathcal S}$ and $W$ is an open subset of $\t$ is a base of topology on $\Sigma$. Suppose indeed that $[t,x]$ belongs to $S(s_1)W_1\cap S(s_2)W_2$ where $s_1,s_2\in {\mathcal S}$ and $W_1,W_2$ are open subsets of $\t$. Using the relation $S(s_i)\theta W_i=S(s_ic_\theta)W_i$, where $\theta\in\t$ and $c_\theta$ is the constant function $\theta$, we may assume that $W_1$ and $W_2$ contain $1$ and that $[t,x]=[s_1,x]=[s_2,x]$. Then, we can find $s\in{\mathcal S}$ and $c\in{\mathcal T}(X)$ such that $s\le s_1$, $sc\le s_2$ and $[t,x]$ belongs to $S(s)W_1\cap S(sc)W_2$. Let $W_3$ be an open neighborhood of $1$ such that $\overline{W_3}W_3\subset W_2$. By continuity of $c$, there exists a compact open neighborhood $V$ of $x$ such that $c(V)\subset W_3$. Then, for $y\in V$ and $\theta\in W_3$, $[s,y]\theta=[sc,y]\overline{c(y)}\theta$ belongs to $S(sc)W_2$. Denoting by $e$ the characteristic function of $V$, we have
 $$[t,x]\in S(se)(W_1\cap W_3)\subset S(s)W_1\cap U(tc)W_2$$
 as desired. The proof that this topology is locally compact and turns $\Sigma$ into a topological groupoid is left to the reader. We observe that $p:\Sigma\to G$ is open and that $j(X\times\t)=p^{-1}(G^{(0)})$ is an open subset of $\Sigma$. Let us check that the continuous bisections of $\Sigma$ are exactly the subsets $S(s)$, where $s\in{\mathcal S}$. Since the restriction of $p:\Sigma\to G$ to $S(s)$ is a homeomorphism of $S(s)$ onto $S(p(s))$, $S(s)$ is a continuous bisection of $\Sigma$. Conversely, let $S$ be a continous bisection of $\Sigma$. Then $p(S)$ is a continuous bisection of $G$. Since $G$ is \'etale, $p(S)$ is an open bisection of $G$. Therefore, there exists $s\in{\mathcal S}$ such that $p(S)=S(p(s))$. Then $S$ and $S(s)$ have the same homeomorphic image in $\G$. By composition, this gives a homeomorphism $h:S(s)\to S$ which is compatible with the domain and range maps. Since $G=\Sigma/\t$ and the action of $\t$ is free, for all $x\in d(s)$ there exist $c(x)\in\t$ such that $h([s,x])=[s,x]c(x)$. Since $c(x)=[s,x]^{-1}h([s,x])$, $c$ is continuous. We can write $h([s,x])=[sc,x]$ and $S=S(sc)$. The map $s\mapsto S(s)$ is a semigroup isomorphism of $\mathcal S$ onto the ample semigroup of $\Sigma$ and its restriction to ${\mathcal T}(X)$ is the identity. Therefore, it is an isomorphism of the twisted Boolean inverse semigroup $(\G,{\mathcal S})$ onto $({\rm Bis}(G), {\rm Bis}(\Sigma))$.
\end{proof}
With the notation of the proposition, we say that $(G,\Sigma)$ is the {\it twisted groupoid of germs} of $(\G,{\mathcal S})$.

\begin{thm}\label{twisted equivalence} The constructions of the groupoid of germs of a Boolean inverse semigroup and of the ample semigroup of a Boolean groupoid are functorial and give an equivalence of the category of twisted Boolean inverse semigroups and the category of twisted Boolean groupoids.
\end{thm}

\begin{proof}  \vskip2mm
 Let $(H,\Lambda)$ and $(G,\Sigma)$ be twisted Boolean groupoids and suppose that there is a $\t$-equivariant multiplier action of $\Lambda$ on $\Sigma$ with proper anchor $\rho^{(0)}$. Passing to the quotient, it gives a multiplier action of $H$ on $G$ with the same anchor. As in \thmref{equivalence}, we have a Boolean inverse semigroup homomorphism $j:{\rm Bis}(H)\to {\rm Bis}(G)$. The same formula
\[j_1(S)\defequal S\Sigma^{(0)}=\{l (x,1): l\in S, x\in G^{(0)}, \rho^{(0)}(x)=d(l)\},\]
where we identify $(x,1)\in G^{(0)}\times\t$ and its image in $\Sigma$,
defines a Boolean inverse semigroup homomorphism $j_1:{\rm Bis}(\Lambda)\to {\rm Bis}(\Sigma)$. It is readily checked that the pair $(j,j_1)$ is an arrow from $({\rm Bis}(H), {\rm Bis}(\Lambda))$ to $({\rm Bis}(G), {\rm Bis}(\Sigma))$. 
\vskip2mm
Let $(j,j_1)$ be an arrow form $({\mathcal H},{\mathcal S}_{\mathcal H})$ to $({\mathcal G},{\mathcal S})$. We let $(H,\Lambda)$ and $G,\Sigma)$ be the corresponding twisted groupoids of germs. Then, $j_1$ defines a multiplier action of $\Lambda$ on $\Sigma$ through the formula
\[[s,\rho_0\circ r(\varphi_t(x))][t,x]=[j_1(s)t,x]\]
where $s\in{\mathcal S}_{\mathcal H}$, $t\in{\mathcal S}$, $x\in{\rm dom}(t)$ and the anchor $\rho_0$ is the dual map of $j^{(0)}:{\mathcal H}^{(0)}\to\G^{(0)}$. It can be checked that its restriction to ${\mathcal T}_{{\mathcal H}^{(0)}}$ is the natural extension of $j^{(0)}$. The quotient multiplier action of $H$ on $G$ is the multiplier action defined by $j$. Again we skip the easy but tedious task to check that these constructions respect the composition of arrows. 
\vskip2mm
The functors Bis and Germ give an equivalence of categories. We have already seen that the twisted Boolean inverse semigroup $(\G,{\mathcal S})$ is isomorphic to the ample inverse semigroup of its twisted groupoid of germs.  On the other hand, let $(G,\Sigma)$ be a twisted Boolean groupoid. Denote by $(\G,{\mathcal S})$ its twisted ample semigroup $({\rm Bis}(G),{\rm Bis}(\Sigma))$. Let $(\underline G,\underline \Sigma)$ be the twisted groupoid of germs of $(\G,{\mathcal S})$. An element of $\underline \Sigma$ is of the form $[S,x]$ where $S\in{\rm Bis}(\Sigma)$ and $x\in{\rm dom}(S)$. The map $f:[S,x]\mapsto Sx$ of $\underline \Sigma$ into $\Sigma$ is a $\t$-equivariant groupoid isomorphism. By construction (see the proof of \propref{isgtwist to gpdtwist}), the family $\{{\underline S}W\}$, where $S\in{\mathcal S}$, $\underline S=\{[S,x], x\in{\rm dom}(S)\}$ and $W$ is an open subset of $\t$, is a base of the topology of $\underline\Sigma$. Let us show that the family $\{SW\}$, where $S\in{\rm Bis}(\Sigma)$ and $W$ is an open subset of $\t$ is a base of topology of $\Sigma$. First, $SW$ is an open subset of $\Sigma$, because it is an open subset of $S\t$, which is itself open in $\Sigma$. Let $\sigma\in\Sigma$ and $U$ open neighborhood of $\sigma$. Let $S'$ be a compact open bisection of $G$ containing $p(\sigma)$ and contained in $p(U)$. As we have seen in the proof of \propref{isgtwist to gpdtwist}, there exists a compact continuous bisection $S$ of $\Sigma$ such that $p(S)=S'$. Then $U$ is an open subset of $S\t$. The map $h:(x,\theta)\mapsto (xS)\theta$ is a homeomorphism of $r(S)\times\t$ onto $S\t$. Let $(x_0,\theta_0)=h^{-1}(\sigma)$. There exists a compact open neighborhood $V$ of $x_0$ and an open neighborhood $W$ of $\theta_0$ such that $V\times W$ is contained in $h^{-1}(U)$. Then $(VS)W$ contains $\sigma$ and is contained in $U$. Since $f(\underline S W)=SW$, $f$ is a homeomorphism.

\end{proof}

\begin{rem}
 Twists over groupoids are a particular case of Fell bundles over groupoids. The relation between Fell bundles over groupoids and Fell bundles over inverse semigroups is studied in \cite{sie:fell} and in \cite{be:fell}.
\end{rem}

\section {Representations of Boolean inverse semigroups}

We first give the definition of a representation of a twisted Boolean inverse semigroup. These representations are called additive in \cite[page 116]{pat:gpd} in the untwisted case.

\begin{defn}
 A representation of a twisted Boolean inverse semigroup $(\G,{\mathcal S})$ with kernel ${\mathcal T}$ and idempotent Boolean algebra $\G^{(0)}=\B(X)$ in a Hilbert space $\mathcal H$ is a map $\pi$ of $\mathcal S$ into the space $B(\mathcal H)$ such that
 
\begin{enumerate}
 \item $\pi(st)=\pi(s)\pi(t)$ for all $s,t\in{\mathcal S}$;
 \item $\pi(s^*)=\pi(s)^*$ for all $s\in{\mathcal S}$;
 \item  $\pi(s\vee t)=\pi(s)+\pi(t)$ for all orthogonal pair $(s,t)$ in $\mathcal S$;
 \item $\pi(0)=0$ and
 \item $\pi$ and $\tilde\pi^{(0)}$, where $\pi^{(0)}$ is the restriction of $\pi$ to $\B(X)$ and $\tilde\pi^{(0)}$ is the extension to $C_0(X)$, agree on the kernel ${\mathcal T}$ of the twist. 
\end{enumerate}
\end{defn}

As expected, representations of $(\G,{\mathcal S})$ correspond bijectively to representations of the $*$-algebra $C_c(G,\Sigma)$, where $(G,\Sigma)$ is the twisted groupoid of germs of $(\G,{\mathcal S})$. This result appears under various guises in the literature  (see for example \cite[Theorem 2.13]{be:fell} and \cite[Theorem 17.13]{ep:char}). We give here an elementary proof. The main point to check is that, given a representation $\pi$ of $(\G,{\mathcal S})$ and $f=\sum_{i\in I}\Delta_{S_i} b_i$ in $C_c(G,\Sigma)$, where the notation is given in section 1.3, the obvious formula $\tilde\pi(f)=\sum_{i\in I}\pi(S_i)\tilde\pi^{(0)}( b_i)$ makes sense. This is done by decomposing the family $(S_i)_{i\in I}$ into a disjoint family. When $G$ is not Hausdorff, this requires to enlarge $\G$.
\vskip2mm

A representation $\tilde\pi^{(0)}$ of $C_0(X)$ can be extended into a representation, still denoted by $\tilde\pi^{(0)}$, of its bidual $C_0(X)''$. In particular, we can extend $\tilde\pi^{(0)}$ to the Boolean algebra ${\mathcal U}(X)$ of universally measurable subsets of $X$ and to the C*-subalgebra $U(X)$ of universally integrable functions on $X$. Given a Boolean groupoid $G$ with unit space $X$, we define the {\it enlarged ample semigroup} $\G_U=\G{\mathcal U}(X)$ as the inverse semigroup of bisections of the form $Se$, where $S$ is a compact open bisection of $G$ and $e\in{\mathcal U}(X)$. It is indeed an inverse semigroup because the partial automorphism $\alpha_S$ of $C_0(X)$ defined by $S$ extends to a partial automorphism of $U(X)$. Given a twisted Boolean groupoid $(G,\Sigma)$ with unit space $X$, we define similarly the inverse semigroup ${\mathcal S}_U={\mathcal S}{\mathcal U}(X)$ whose elements are the bisections of the form $Se$, where $S$ is a compact continuous bisection of $\Sigma$ and $e\in{\mathcal U}(X)$. When $G$ is not Hausdorff, we need to enlarge similarly the $*$-algebra $C_c(G,\Sigma)$. Its elements will still be sections $\underline f$ of the hermitian line bundle $(\Sigma\times{\bf C})/\t$ over $G$, which we may view as complex-valued functions $f$ on $\Sigma$ satisfying $f(\theta\sigma)=\overline\theta f(\sigma)$ for $\theta\in\t$; $f$ and $\underline f$ are related by $\underline f(\sigma')=[\sigma, f(\sigma)]$, where $\sigma'$ is the image of $\sigma$ by the quotient map $p:\Sigma\to G$. Given $S\in {\mathcal S}_u$, we define the section $\underline\Delta_S$ by  $\underline\Delta_S(\sigma')=[Sd(\sigma'),1]$ if $\sigma'\in\dot S$ and $\underline\Delta_S(\sigma')=0$ if $\sigma'\notin\dot S$. Equivalently $\Delta_S(\sigma)=\overline\theta$ if $\sigma=\theta(Sd(\sigma))$, where $\theta\in\t$, and $\Delta_S(\sigma)=0$ if $\sigma$ does not belong to $\t S$. The elements of our enlarged $*$-algebra, denoted by $U(G,\Sigma)=C_c(G,\Sigma)U(X)$, are finite sums  $f=\sum_i\Delta_{S_i} b_i$ where $S_i$ is a compact continuous bisection of $\Sigma$ and $b_i\in U(X)$. The definition of the convolution product and of the involution is the same as in Section 1.3 and one proves similarly that it is a $*$-algebra.

\begin{defn} \label{refine}We say that a finite family $({T_j}')_{j\in J}$ of bisections of a groupoid $G$ {\it refines} another finite family $({S_i}')_{i\in I}$ of bisections if for every $i\in I$, ${S_i}'$ is a disjoint union $\vee_{j\in J(i)}{T_j}'$.
\end{defn}

\begin{lem}\label{atoms}
Let $({S_i}')_{i\in I}$ be a finite family of bisections of a groupoid $G$. We define
\begin{enumerate}
 \item for $\gamma\in G$, $\omega(\gamma)=\{i\in I: \gamma\in {S_i}'\}$;
\item $\Omega=\{\omega(\gamma): \gamma\in G\}\setminus\emptyset$;
 \item for $\omega\in \Omega$, ${T_\omega}'=\{\gamma\in G: \omega(\gamma)=\omega\}$.
\end{enumerate}
Then $({T_\omega}')_{\omega\in\Omega}$ is a disjoint family of bisections which refines the family $({S_i}')_{i\in I}$. Moreover, if $G$ is Boolean and the ${S_i}'$'s are compact open, then the $d({T_\omega}')$'s are locally closed subsets of $G^{(0)}$. Therefore, the ${T_\omega}'$'s belong to $\G_U$.
\end{lem}

\begin{proof} Left to the reader.
 Note that ${T_\omega}'=\cap_{i\in\omega} {S_i}'\setminus \cup_{i\notin\omega} {S_i}'$ is a locally closed subset of $S_i$. Therefore $d(T_\omega')$, which is its image by the homeomorphism $d_{|S_i}$ is a locally closed subset of the open set $d(S_i)$, hence it is a locally closed subset of $G^{(0)}$. Moreover, locally closed subsets of a locally compact Hausdorff space are universally measurable.
 \end{proof}

\begin{lem}\label{2} Let $(G,\Sigma)$ be a twisted Boolean groupoid with unit space $X$. Then
every $f\in U(G,\Sigma)$ can be written $f=\sum_{j\in J}\Delta_{T_j}c_j$, where $({T_j}')_{j\in J}$ is a disjoint family in $\G_U$ and $c_j\in U(X)$. Moreover, if $f$ has another expression $f=\sum_{j\in J}\Delta_{T_j}d_j$, then $c_j$ and $d_j$ agree on $d(T_j)$.
\end{lem}

\begin{proof} Let $f=\sum_{i\in I}\Delta_{S_i} b_i$ be an arbitrary expression of $f\in M_c(G,\Sigma)$. We apply \lemref{atoms} to the finite family $({S_i}')_{i\in I}$. We obtain the finite disjoint family $({T_\omega}')_{\omega\in \Omega}$ in $\G_u$. We pick for each $\omega\in\Omega$ an index $i(\omega)\in I$ such that $i(\omega)\in\omega$. Then, we can write $f=\sum_{\omega\in\Omega}\Delta_{S_{i(\omega)}}b_\omega$, where $b_\omega(x)=f(S_{i(\omega)}x)$ for $x\in d(S_\omega)$ and $b_\omega(x)=0$ otherwise. Thus, we have the existence of a disjoint expression of $f$. If $f=\sum_{j\in J}\Delta_{T_j}c_j$ where $({T_j}')_{j\in J}$ is a disjoint family in $\G_u$, then for all $j\in J$ and all $x\in d(T_j)$, $f(T_jx)=\Delta_{T_j}(T_jx)c_j(x)=c_j(x)$. This gives the required uniqueness.
 
\end{proof}

\begin{lem}\label{1} Let $(G,\Sigma)$ be a twisted Boolean groupoid. Let $f=\sum_{i\in I}\Delta_{S_i} b_i=\sum_{j\in J}\Delta_{T_j} c_j$ be two expressions of $f\in U(G,\Sigma)$. Assume that the family $({T_j}')_{j\in J}$ of $\G''$ is disjoint and refines the family $({S_i}')_{i\in I}$. Then we have the equality
\[\sum_{i\in I}\pi(S_i)\tilde\pi^{(0)}( b_i)=\sum_{j\in J}\pi(T_j) \tilde\pi^{(0)}(c_j)\]
for every representation $\pi$ of $(G,\Sigma)$.
 \end{lem}

\begin{proof} Let $\mathcal T$ be the kernel of the quotient map $p:{\mathcal S}\to\G$ of the twisted ample semigroup of $(G,\Sigma)$. By assumption, for every $i\in I$, there exists $J(i)\subset J$ such that
${S_i}'=\vee_{j\in J(i)}{T_j}'$. Then $p(S_i d(T_j))=p(T_j)$. Therefore, there exist $t_{(i,j)}\in{\mathcal T}$ such that $S_i=\vee_{j\in J(i)}T_jt_{(i,j)}$. This implies that $\Delta_{S_i}=\sum_{j\in J(i)}\Delta_{T_j}t_{(i,j)}$. Therefore, 
\begin{align*}
\sum_{i\in I}\Delta_{S_i} b_i=&\sum_{i\in I}\big(\sum_{j\in J(i)}\Delta_{T_j}t_{(i,j)}\big) b_i\cr
=&\sum_{(i,j)\in K}\Delta_{T_j}t_{(i,j)} b_i\cr
=&\sum_{j\in J}\Delta_{T_j}\big(\sum_{i:j\in J(i)}t_{(i,j)} b_i\big).\cr
\end{align*}
Since the family $({T_j}')_{j\in J}$ is disjoint, the uniqueness in \lemref{2} says that $c_j$ and $\sum_{i:j\in J(i)}t_{(i,j)} b_i$ agree on $d(T_j)$. On the other hand, we have $\pi(S_i)=\sum_{j\in J(i)}\pi(T_j)\tilde\pi^{(0)}(t_{(i,j)})$. Therefore,
\begin{align*}
\sum_{i\in I}\pi(S_i)\tilde\pi^{(0)}( b_i)=&\sum_{i\in I}\big(\sum_{j\in J(i)}\pi(T_j)\tilde\pi^{(0)}(t_{(i,j)})\big)\tilde\pi^{(0)}( b_i)\cr
=&\sum_{(i,j)\in K}\pi(T_j)\tilde\pi^{(0)}(t_{(i,j)})\tilde\pi^{(0)}( b_i)\cr
=&\sum_{(i,j)\in K}\pi(T_j)\tilde\pi^{(0)}(t_{(i,j)} b_i)\cr
=&\sum_{j\in J}\pi(T_j)\tilde\pi^{(0)}(\sum_{i:j\in J(i)}t_{(i,j)} b_i)\cr
=&\sum_{j\in J}\pi(T_j)\tilde\pi^{(0)}(c_j).\cr
\end{align*}
\end{proof}

\begin{cor}\label{welldefined} Let $(G,\Sigma)$ be a twisted Boolean groupoid. Let $f=\sum_{i\in I}\Delta_{S_i} b_i=\sum_{j\in J}\Delta_{T_j} c_j$ be two expressions of $f\in U(G,\Sigma)$. Then we have the equality
\[\sum_{i\in I}\pi(S_i)\tilde\pi^{(0)}( b_i)=\sum_{j\in J}\pi(T_j) \tilde\pi^{(0)}(c_j)\]
for every representation $\pi$ of $(G,\Sigma)$.
 \end{cor}
 
 \begin{proof}
 Applying \lemref{atoms} to the union of $({S_i}')_{i\in I}$ and $({T_j}')_{j\in J}$, we obtain a disjoint family $( R'_k)_{k\in K}$ which refines both $({S_i}')_{i\in I}$ and $({T_j}')_{j\in J}$. Then, as in the proof of \lemref{2},  we can write $f$ under the form $f=\sum_{k\in K}\Delta_{R_k}d_k$. According to \lemref{1},
 \[\sum_{i\in I}\pi(S_i)\tilde\pi^{(0)}( b_i)=\sum_{k\in K}\pi(R_k) \tilde\pi^{(0)}(d_k)=\sum_{j\in J}\pi(T_j) \tilde\pi^{(0)}(c_j).\]
\end{proof}

\begin{thm}\label{twisted rep}
 Let $(\G,{\mathcal S})$ be a twisted Boolean inverse semigroup and let $(G,\Sigma)$ be its groupoid of germs. Then there is a one-to-one correspondence between the representations $\pi$ of $(\G,{\mathcal S})$  and the representations $\tilde\pi$ of $C_c(G,\Sigma)$ such that for $S\in{\mathcal S}$, $\pi(S)=\tilde\pi(\Delta_S)$.
\end{thm}

\begin{proof} We identify ${\mathcal S}={\rm Bis}(\Sigma)$ and $\G={\rm Bis}(G)$. Let $\tilde\pi$ be a representation of $C_c(G,\Sigma)$. According to \lemref{map u}, $\pi(S)=\tilde\pi(\Delta_S)$ defines a representation of the twisted Boolean inverse semigroup $(\G,{\mathcal S})$. Conversely, let $\pi$ be a representation of the twisted Boolean inverse semigroup $(\G,{\mathcal S})$. We extend $\pi^{(0)}$ to $C_0(X)$ (and to its bidual $C_0(X)''$), where $X$ is the spectrum of the Boolean algebra $\G^{(0)}$. Every $f\in C_c(G,\Sigma)$ can be written under the form $f=\sum_{i=1}^m \Delta_{S_i}b_i$ where $S_i\in{\mathcal S}$ and $b_i\in C_c(d(S_i))$. According to \corref{welldefined}, $\tilde\pi(f)=\sum_{i=1}^m \pi(S_i)\tilde\pi^{(0)}(b_i)$ is well defined. It is then straightforward to check that $\tilde\pi$ is a $*$-representation of $C_c(G,\Sigma)$. Indeed $f=\sum_{i=1}^m \Delta_{S_i}b_i$ and $g=\sum_{j=1}^n \Delta_{T_j}c_j$, then $f*g=\sum_{(i,j)\in K}\Delta_{S_iT_j}d_{(i,j)}$ where $K=\{1,\ldots,m\}\times\{1,\ldots,n\}$ and $d_{(i,j)}=\alpha_{T_j^{-1}}(b_i)c_j$ with $\alpha_{T^{-1}}(e)=T^{-1}eT$ for $e\in\B(X)$. The equality $\tilde\pi(f*g)=\tilde\pi(f)\tilde\pi(g)$ results from the equality
$\tilde\pi^{(0)}(b)\pi(T)=\pi(T)\tilde\pi^{(0)}(\alpha_{T^{-1}}(b))$
which suffices to check when $b=e\in\B(X)$. Then, we have indeed
\[\pi(T)\tilde\pi^{(0)}(\alpha_{T^{-1}}(e))=\pi(T)\pi(T^{-1}eT)=\pi(e)\pi(T)=\tilde\pi^{(0)}(e)\pi(T).\]
The equality $\tilde\pi(f^*)=\tilde\pi(f)^*$ results from the same equality.
 \end{proof}


\begin{thebibliography}{10}


\bibitem{adp:II1}  C.~Anantharaman-Delaroche, S.~Popa, \textit{An introduction to $\rm{II}_1$ factors}, preprint.

\bibitem{adr:amenable}  C.~Anantharaman-Delaroche, J.~Renault, \textit{Amenable groupoids}, L'Enseignement Math\'ematique, Gen\`eve 2000.

\bibitem {acclmr} B.~Armstrong, L.O.~Clark, K.~Courtney, Y.F.~Lin, K.~McCormick, J.~Ramagge.\textit{Twisted Steinberg algebras}, J. Pure Appl. Algebra \textbf{226} (2022) no.3, 205--252.

\bibitem{accclmrss} B.~Armstrong, G.~ Castro, L.O.~Clark, K.~Courtney, Y.F.~Lin, K.~McCormick, J.~Ramagge, A.~Sims, B.~Steinberg,  \textit{Reconstruction of twisted Steinberg algebras}, Int. Math. Res. Not. IMRN(2023), no. 3, 2474--2542.

\bibitem{abcclmr} B.~Armstrong, J.~ Brown, L.O.~Clark, K.~Courtney, Y.F.~Lin, K.~McCormick, J.~Ramagge. \textit{The local bisection hypothesis for twisted groupoid C*-algebras}, arXiv:math.OA/2307.13814v2.

\bibitem{bl:cartan II} S.~Barlak and X.~Li  \textit{Cartan subalgebras and the UCT problem, II}, Math. Ann. \textbf{378} (2020), no. 1-2, 255--287.

\bibitem {ber:AW*} S.~K.~Berberian, \textit{Baer $*$-Rings}, Die Grundlehren der mathematischen Wissenschaften, Band 195, Springer-Verlag New-York-Berlin, 1972.

\bibitem{bbki:ta} N.~Bourbaki, \textit{El\'ements de Math\'ematique, Topologie Alg\'ebrique, Chapitres 1 \`a 4}, Springer-Verlag Berlin-Heidelberg, 2016.

\bibitem{bfpr} J.~Brown, A.~Fuller,D.~Pitts and S.~Reznikoff, \textit{Graded C*-algebras and twisted groupoid C*-algebras}, New York J. Math. 27 (2021), 205–252.

\bibitem{bs:morphisms} M.~Buneci and P.~Stachura, \textit{Morphisms of locally compact groupoids endowed with Haar systems}, arXiv:math.OA/0511613v2.

\bibitem{bem:isa} A.~Buss, R.~Exel and R.~Meyer, \textit{Inverse semigroup actions as groupoid actions}, arXiv:math.1104.0811v4.

\bibitem{be:fell} A.~Buss and R.~Exel, \textit{Fell bundles over inverse semigroups and twisted \'etale groupoids}, J. Operator Theory, \textbf{67}
(2012), 153--205.

\bibitem{bm:isg} A.~Buss and R.~Meyer, \textit{Inverse semigroup actions on groupoids}, Rocky Mountain J. of Math., \textbf{47}
(2017), 53--159.

\bibitem {cpz:bimodules}  J.~Cameron, D.~Pitts and V.~Zarikian, \textit{Bimodules over Cartan MASAs in von Neumann algebras, norming algebras, and Mercer's theorem}, New York J. Math.  \textbf{19} (2013), 455--486.

\bibitem {dix:stone}  J.~Dixmier, \textit{Sur certains espaces consid\'er\'es par M.~H.~Stone}, Summa Brasil. Math. 2 (1951), 151--182.

\bibitem {dfp:cartan}  A.~Donsig, A.~Fuller, D.~Pitts, \textit{Von Neumann algebras and extensions of semigroups}, Proc. of the Edinburgh Math. Soc. \textbf{60}, issue 1 (2017), 57--97.

\bibitem {dwz}  A.~Duwenig, D.~Williams and J.~Zimmerman, \textit{Renault's j-map for Fell bundle C*-algebras}, J. Math. Anal. Appl. 516 (2022), no. 2, Paper No. 126530, 15 pp.

\bibitem {exe:isg}  R.~Exel, \textit{Inverse semigroups and combinatorial C*-algebras}, Bull. Braz. Math. Soc.(N. S.) \textbf{39} (2008), 191--313.

\bibitem {exe:nh}  R.~Exel, \textit{Non-Hausdorff étale groupoids}, Proc. Amer. Math. Soc. \textbf{139} (2011), no. 3, 897--907.
 

\bibitem{ep:char} R.~Exel and D.~Pitts, \textit{Characterizing groupoid C*-algebras of non-Hausdorff \'etale groupoids}, Lecture Notes in Mathematics, Vol.~{\bf 2306},
Springer Cham, 2022.

\bibitem {fm:relations I} J.~Feldman and C.~Moore, \textit{Ergodic equivalence relations,
cohomologies, von Neumann algebras, I}, Trans. Amer. Math. Soc.  \textbf{234} (1977),
289--324.

\bibitem {fm:relations II} J.~Feldman and C.~Moore, \textit{Ergodic equivalence relations,
cohomologies, von Neumann algebras, II}, Trans. Amer. Math. Soc.  \textbf{234} (1977),
325--359.

\bibitem {gh:boole} S.~Givant and P.~Halmos, \textit{Introduction to Boolean algebras}, Undergraduate Texts in Mathematics, Springer-Verlag New-York-Berlin, 2009.

\bibitem {gla:joinings} E.~Glasner, \textit{Ergodic theory via joinings}, Mathematical Surveys and Monographs, Vol.~{\bf 582}
American Mathematical Society, Providence, 2003.

\bibitem {gut:bundles} A. Gutman, \textit{Banach bundles in the theory of lattice-normed spaces. I. Continuous Banach bundles}, Siberian Adv. Math. \textbf{3} (1993), no. 3, 1--55.

\bibitem {hah:regular} P.~Hahn, \textit{The regular representation of measure groupoids}, Trans. Amer. Math. Soc. \textbf{242} (1978), 35--72.


\bibitem {hal:measure} P.~Halmos, \textit{Measure theory}, The university series in higher mathematics, Van Nostrand Princeton, 1968.

\bibitem {hol:I} R.~Holkar \textit{Topological construction of  C*-correspondences for groupoid  C*-algebras}, J. Operator Theory \textbf{77} (2017), no. 1, 217--241.

\bibitem {hol:II} R.~Holkar \textit{Composition of topological correspondences}, J. Operator Theory \textbf{78} (2017), no. 1, 89--117.

\bibitem {ks:regular} M.~Khoshkam, G.~Skandalis \textit{Regular representation of groupoid C*-algebras and applications to inverse semigroups}, J. reine angew. Math. \textbf{546} (2002), 47--72.


\bibitem {kklru:essential}  M.~Kennedy, S.-J.~Kim, X.~Li, S.~Raum, D.~Ursu, \textit{The intersection property for essential groupoid C*-algebras}, math arXiv: 2107.03980v2

\bibitem {kri:dimension} W.~Krieger, \textit{On a dimension for a class of homeomorphism groups}, Math. Ann. \textbf{252} (1979/1980) no 2, 87--95.

\bibitem {kum:diagonals} A.~Kumjian, \textit{On C$^*$-diagonals}, Can.~J.
Math., Vol.~{XXXVIII},\textbf{4} (1986), 969--1008.

\bibitem {law:inverse} M.V.~Lawson, \textit{Inverse Semigroups}, World Scientific, Singapore, 1998. 

\bibitem {law:stone} M.V.~Lawson, \textit{A non-commutative generalization of Stone duality}, J. Aust. Math. Soc.  
\textbf{88} (2010), 385--404.

\bibitem {law:stone2} M.V.~Lawson, \textit{
 Non-commutative Stone duality: inverse semigroups, topological groupoids and C*- algebras}, Internat. J. Algebra Comput. 22 (2012), no. 6, 1250058, 47.
 
 \bibitem {law:stone3} M.V.~Lawson, \textit{
 Non-commutative Stone duality}, preprint, arXiv:2207.02686v2.
 
 \bibitem {law:isg3} M.V.~Lawson, \textit{Primer on inverse semigroups I}, arXiv:2006.01628.

\bibitem {ll:pseudogroups} M.V.~Lawson and D. Lenz, \textit{Pseudogroups and their \'etale groupoids}, Adv. Math. \textbf{244} (2013), 117--170.

\bibitem {mz:categories} R..~Meyer and C.~Zhu, \textit{Groupoids in categories with pretopology}, Theory Appl. Categ. \textbf{30} (2015), 1906--1998.

 \bibitem {pat:pb} A.~Paterson, \textit{Inverse Semigroups, groupoids and a problem of J. Renault} in {\it Algebraic methods in operator theory}, ed. R. Curto and P. Jorgensen,  Birkh\"auser, Boston, 1993.
 
 \bibitem {pat:gpd} A.~Paterson, \textit{Groupoids, Inverse Semigroups, and their Operator Algebras}, Progress in Mathematics, vol. \textbf{170},  Birkh\"auser, Boston, 1999.

\bibitem{pit:unit} D.~Pitts \textit{Normalizers and approximate units for inclusions of C*-algebras}, preprint, arXiv:2109.00856.
 
 \bibitem {qs:C*-actions} J.~Quigg and N.~Sieben, \textit{C*-actions of $r$-discrete groupoids and inverse semigroups}, J. Austral. Math. Soc. Ser A \textbf{66} (1999), 143--167.
 
\bibitem {raa:cartan} A.~Raad, \textit{A Generalization of Renault’s Theorem for Cartan Subalgebras}, Proc. Amer. Math. Soc. \textbf{150} (2022), 4801--4809.

 \bibitem {ren:approach}  J.~Renault, \textit{A groupoid approach to
$C^*$-algebras}, Lecture Notes in Mathematics, Vol.~{\bf 793},
Springer-Verlag Berlin, Heidelberg, New York, 1980.

\bibitem{ren:rep} J.~Renault:  {\it Repr\'esentations des produits crois\'es d'alg\`ebres de
groupo\"{\i}des}, J. Operator Theory, \textbf{18}
(1987), 67--97.

\bibitem {ren:cartan} J.~Renault, \textit{Cartan subalgebras in C$^*$-algebras}, Irish Math. Soc. Bulletin \textbf{61} (2008), 29--63.

\bibitem {ren:rieffel} J.~Renault, \textit{Induced representations and groupoids}, SIGMA Symmetry Integrability Geom. Methods Appl.10 (2014), Paper 057, 18 pp.

\bibitem {ren:twisted extensions} J.~Renault, \textit{Abelian twisted groupoid extensions}, J. of Operator Theory \textbf{89:1} (2023), 249--283.

\bibitem{sie:fell} N.~Sieben \textit{Fell bundles over r-discrete groupoids and inverse semigroups}, preprint.

\bibitem {ste:isga} B.~Steinberg, \textit{A groupoid approach to discrete inverse semigroup algebras}, Adv. Math. \textbf{223}  no. 2 (2010), 689--727.

\bibitem{tay:functoriality} J.~Taylor \textit{Functoriality for groupoid and Fell bundle C*-algebras}, arXiv:2310.03126.

\bibitem {weh:boolean}  F.~Wehrung, \textit{Refinement monoids, equidecomposability types, and Boolean inverse semigroups}, Lecture Notes in Mathematics, Vol.~{\bf 2188}
Springer-Verlag Berlin, Heidelberg, New York, 2017.

\bibitem {wid:embedding} H.~Widom, \textit{Embedding in algebras of type I}, Duke Math. J.  \textbf{23} (1956), 309--324.


\end{thebibliography}
\end{document}